\documentclass[11pt]{amsart}

\input{preamble}

\title{Universal Cancellations in Uniform Random Waves} 
\author[L.~Gass, D.~Marinucci, G.~Peccati, F.~Pistolato, M.~Stecconi]{Louis Gass, Domenico Marinucci, Giovanni Peccati, Francesca Pistolato, Michele Stecconi}
\address{Department of Mathematics, Luxembourg University \\louis.gass@uni.lu}
\address{Department of Mathematics, University of Rome Tor Vergata and INdAM \\
marinucc@mat.uniroma2.it}
\address{Department of Mathematics, Luxembourg University \\giovanni.peccati@uni.lu}
\address{Department of Mathematics, Luxembourg University \\francesca.pistolato@uni.lu}
\address{Department of Mathematics, Luxembourg University \\michele.stecconi@uni.lu}

\date{\today}
\begin{document}
\maketitle
\begin{abstract}
A vast literature over the past fifteen years has been devoted to the study of the geometric properties of Gaussian random waves. In this work, we investigate the geometric behavior of \emph{uniform random waves}, a much less studied non-Gaussian model in which the $L^2$ norm is constrained to be exactly equal to one in every realization (a normalization that is natural from the standpoint of quantum mechanics). We show that this norm-constrained formulation has deep consequences for the universality of the so-called \emph{Berry's cancellation phenomenon}, as well as for novel high-frequency asymptotic variance estimates. These effects manifest themselves in both local geometric functionals, such as the Lipschitz--Killing curvatures, and global ones, such as the number of connected components above a fixed threshold. A key byproduct of our analysis is a new explicit relation between Hermite expansions and spherical harmonic decompositions for $0$-homogeneous functionals of Gaussian vectors, which enables a systematic chaos-based analysis of non-Gaussian random waves.

\itemize
\item{AMS subject classification (2020 MSC): 60G60, 33C55, 53C65, 60D05} 
\item {Keywords: Random Fields, Random Waves, Wiener Chaos Expansions, Lipschitz-Killing Curvatures, Betti Numbers}

\end{abstract}
\setcounter{tocdepth}{2}
\tableofcontents 

\newcommand{\lebf}{\tilde{f}}
\newcommand{\I}{\mathbb{I}}

\section{Uniform Random Waves}\label{sec:urw}

\subsection{Background and Motivations}

The geometric investigation {of} the excursion sets of random waves {(that is, of random (approximate) Laplace eigenfunctions)} on manifolds has been a very active area of research over the last fifteen years. In particular, remarkable interest has been drawn by the analysis of their zero sets, typically corresponding to a union of random smooth {\it nodal lines} or {\it(hyper)surfaces}; see \cite{wigman_survey_2024} for an overview. 

\smallskip 

Such a stream of research was initiated by a collection of groundbreaking conjectures in mathematical physics by Michael V. Berry, mostly rooted in his seminal 1977 contribution \cite{Berry1977}. The crucial insight in Berry’s work, nowadays known as {\it Berry’s random wave conjecture}, is that — in the high-energy limit — a deterministic Laplace eigenfunction on a {general} Riemannian manifold with chaotic geodesic flow should behave locally like an isotropic Gaussian random wave on a flat space. 

In recent years, several papers have tried to make this conjecture rigorous and to establish its validity under suitable circumstances: the reader is referred to, e.g. \cite{Abert2018, garcíaruiz2023relation, Ing21} for a characterization of Berry's conjecture in terms of {\it Benjamini-Schramm convergence}, and to \cite{canzani_hanin_2015_scalimkernel, canzani_hanin_2018_cinfscasymp, CH20} for results describing Euclidean isotropic waves as universal local limits of {\it Gaussian monochromatic waves} on Riemannian manifolds without conjugate points. We also mention the recent contribution \cite{Gass2023}, which deals with a class of {\it large-window models} (as opposed to the monochromatic setting mentioned above): it is proved there that a generic deterministic linear combination of Laplace eigenfunctions, chosen from a large window of close eigenvalues and evaluated in the neighborhood of a uniformly random point on a smooth compact manifold, exhibits Gaussian fluctuations in the high-energy limit; in particular, when the window expands at a sufficiently slow rate, these fluctuations are consistent with Berry’s conjecture.

\smallskip

An extensive part of the literature has focused on two-dimensional manifolds --- with particular emphasis on the Euclidean plane, the torus, and the 2-dimensional sphere (see, e.g., \cite{Wigman_2010, Kri13, MRW20, Dierickx2023, NourdinPeccatiRossi2019, MaRossiJFA, PV, NPV, MaWi2014, grotto2023nonlinearfunctionalshyperbolicrandom, smutek}), and further contributions extending the analysis to higher-dimensional settings (see, for instance, \cite{Cam19, Dalmao2021, MauriziaJTP, Not21}). While a substantial portion of this research has concentrated on nodal lengths (or, more generally, nodal volumes), many other geometric functionals have been investigated as well — including {\it excursion areas/volumes}, {\it Euler–Poincaré characteristics}, {\it critical points}, {\it Radon-type transforms}, and topological quantities such as the number of {\it connected components} (see e.g. \cite{ancona2025zeroscriticalpointsgaussian, wigman_survey_2024, Gass2023, GassStecconi2024, LETENDRE2016, BeliaevMcAuleyMuirheadAoP2024, Tacy23}). 

\smallskip 

As demonstrated, for instance, in \cite{Ber02, Cam19, Wig09, MaPeRoWi2015, NourdinPeccatiRossi2019, MRW20, Not21, DNPR_EJP}, most of the functionals mentioned above exhibit a specific and initially unexpected feature in the high-energy regime --- the so-called {\it Berry’s cancellation phenomenon}. This phenomenon is a reduction-of-variance effect that typically occurs for local or semi-local geometric functionals associated with the zero set of a random wave. In particular, at the zero threshold $u=0$, the variance of such functionals has been shown to be smaller by one order of magnitude (up to logarithmic factors) than the standard heuristic prediction, and indeed significantly lower than the variance of the same functionals evaluated at generic non-zero thresholds. The existence of this cancellation was first detected (at a physical level of rigour) in \cite{Ber02}, and then formally established in \cite{Wig09} for nodal lengths of spherical harmonics; it was then shown in \cite{MaPeRoWi2015} (for the torus) and \cite{MRW20, NourdinPeccatiRossi2019} (for the sphere and the Euclidean plane, respectively) that the phenomenon could be explained in terms of the disappearance at level zero of the so-called \emph{second-order chaotic term} in the decomposition of the nodal volume into {\it projections on Wiener chaoses} (see Section \ref{sec:modelmain} for formal definitions). Very recently, it has been proved in \cite{cgv2025StecconiTodino} that an analogous phenomenon occurs for nodal volumes associated with random waves defined on general Riemannian manifolds.

\smallskip 

The second-order chaotic terms evoked above can be shown to be roughly proportional to the (random) \(L^2\) norm of the associated eigenfunctions minus its mean; it is then intuitive that such a term must disappear (or collapse to a boundary term) at the zero level, since the corresponding functionals are invariant under a multiplicative rescaling of the field. The same phenomenon holds for a wide variety of different functionals, including the excursion area (see \cite{MaWi2011excarea,MaWi2014}), the number of critical points \cite{CammarotaMarinucciWigmanJGA2016}, the Euler-Poincaré characteristic \cite{CammarotaM2018}; similar cancellations occur also for non-local or semi-local quantities, such as the number of nodal domains (see \cite{NazarovSodin2009,NazarovSodin2016}), although so far they have not been connected to Wiener chaos decompositions, which are in principle only amenable to analysis for local functionals (see, however, the recent groundbreaking work \cite{mcauley2025limittheoremsnumbersign}, as well as \cite{BeliaevMcAuleyMuirheadAoP2024,BeliaevMcAuleyMuirheadPTRF2020}).

\smallskip

Given the role of normalization factors in the asymptotic behavior of geometric functionals, it is natural to ask whether Gaussianity provides the most appropriate framework for modeling random waves. We note that Gaussianity is a standard choice, strongly motivated both mathematically (by the Central Limit Theorem) and physically (for instance, the Gaussian behavior of certain random fields on ${S^2}$ is a key prediction of the {\it inflationary paradigm} of the standard Cosmological Model, see \cite{Marinucci_Peccati_2011}).  However, random waves originate in quantum mechanics, where eigenfunctions of the Laplace operator model quantum states and their squared moduli represent spatial probability densities. In this setting—formalized in the Dirac–von Neumann axioms (see, e.g. \cite{ReedSimon})—eigenfunctions are always normalized to unit norm, whereas Gaussian random waves do not satisfy this constraint at fixed eigenvalue: although their norms converge to one in the high-frequency limit, exact normalization is not imposed at the level of the model.

\smallskip 

The aim of this paper is therefore to investigate the geometry of random eigenfunctions under the exact unit-norm constraint. For reasons that will become clear later (see Remark \ref{r:uniform}), we refer to such fields as \emph{uniform random waves}. While their behavior coincides with that of Gaussian random waves in the nodal case, significant differences emerge at non-zero thresholds, and the resulting model aligns more naturally with the probabilistic interpretation of eigenfunctions as densities.

\smallskip 
The following list contains an informal description of some of the main achievements of our work. \smallskip \begin{itemize} 
\item[--] We will show that, for uniform random waves, a \emph{static form} of Berry’s cancellation (see Definition \ref{d:berrycancellation}) persists and is not limited to the zero level, but rather occurs at every threshold. Moreover, this cancellation holds for a wide class of geometric functionals—both local and non-local—including situations where the explicit form of their chaos expansion is either unknown or not tractable, such as for topological quantities (e.g., the number of connected components of level sets). See Theorem \ref{thm:Main1}, Corollary \ref{c:splendid1}, Theorem \ref{thm:Main2} and \cref{thm:louis2}. 
\item[--] We will establish a novel correspondence between two alternative decompositions of $0$-homogeneous functionals: one based on Wiener chaos expansions, and the other on high-dimensional spherical harmonics. This correspondence relies on the description of spherical harmonics in terms of traceless tensors, together with new combinatorial identities for Wick products. See Theorem \ref{t:nonmale} and Appendix \ref{A:SHproofs}.

\item[--] We will derive explicit Wiener chaos expansions for several geometric functionals evaluated on the excursion sets of uniform random waves, despite their intrinsically non-Gaussian nature. Our approach exploits the representation of uniform random waves as suitably renormalized Gaussian fields. See Proposition \ref{thm:Main3}, Theorem \ref{t:wienerchaos} and Proposition \ref{prop:secondchaos}. \item[--] The results described above will allow us to derive explicit lower bounds for the asymptotic variance of certain geometric functionals, such as the excursion area (see Proposition \ref{p:lowerbound}), and will also lead us to formulate a collection of natural conjectures concerning the joint behavior of geometric functionals at different thresholds and their high-frequency correlation structure, in the spirit of what was established for partial correlations in \cite{MarinucciRossiAHP2021}.

\end{itemize}

\medskip

Norm-constrained random waves have already appeared in the probability and mathematical physics literature. Our work is, however, the first to investigate their fluctuations from the perspective of Wiener chaos expansions and cancellation phenomena. The recent contributions \cite{FengAdler19, FengXuAdler18, FengYaoAdler2025} focus instead on tail estimates for the suprema of uniform random waves, obtained via embedding techniques. Norm-constrained random fields have also been considered in \cite{BurqLebeau13, Han17, Maples13, Zelditch96}, where they are used to derive generic analytic estimates for Laplace eigenfunctions on Riemannian manifolds.

\medskip

In forthcoming Section \ref{sec:modelmain}, we will provide more details on our model and state our main results in a concise form.

\subsection{Plan of the paper}

The structure of this paper is as follows: in Section \ref{sec:modelmain}
we introduce our model and illustrate our results 
with a high-level overview of the arguments and ideas behind the proofs. The latter are presented in a more complete form in Sections \ref{sec:main1}, \ref{sec:main2}, and \ref{sec:main3}; analytical tools of some independent interest are discussed in Section \ref{sec:magictrick} while further technical details are collected in two Appendices.

\subsection{Acknowledgments}
The research leading to this paper has been supported by:

\smallskip

\begin{itemize}
  \item[--] the PRIN project \emph{Grafia} (CUP: E53D23005530006);
  \item[--] the PRIN Department of Excellence \emph{MatMod@Tov} (CUP: E83C23000330006);
  \item[--] the Luxembourg National Research Fund (Grant: O24/18972745/GFRF).
\end{itemize}

\smallskip

\subsection{Notations}
\subsubsection{Random elements}\label{ss:randin} For the rest of the paper, every random element is assumed to be defined on a common probability space $(\Omega, \mathscr{F}, \P)$, with the symbol $\E$ denoting expectation with respect to $\P$. We write $A\stackrel{d}{=} B$ to indicate that $A,B$ have the same distribution. Given a topological space $\mathscr{X}$, we will write 
\be \label{e:randin}
Z\, \randin \, \mathscr{X}
\ee
to denote the fact that $Z$ is a measurable random element with values in $\mathscr{X}$.

\subsubsection{Notation for integrals, volumes and spheres}\label{subsec:notaintegrals}
Given a Riemannian manifold $(M,g)$ of dimension $d$ and $k\in \N$, we
denote the associated $k$-dimensional Hausdorff measure by $\vol{k}(\cdot)$ and we drop the superscript in the case $k=d$. Moreover, given a Borel function $F\colon M\to \R$, we will use the following shorthand notations: 
\begin{eqnarray}
&& \int_MF(x)dx:=\int_MF(x)\vol{d}(dx),\\ 
&& \vol{}(M):=\int_M1dx,\quad\mbox{and}\quad \fint_MF(x)dx:=\frac{1}{\vol{}(M)}\int_MF(x)dx;\notag
\end{eqnarray}

In particular, if $S\subset M$ is a submanifold of dimension $k$, then $\int_SF(x)dx$ stands for the integral with respect to the $k$-dimensional Hausdorff measure associated with $g$ (or, equivalently, to $g|_S$), unless otherwise specified. 
{Note that, although such notation is not standard, it will be adopted throughout the paper to make our formulas more readable while causing only minimal ambiguities.

We denote by $\mC^\infty(M)$ the space of real-valued smooth functions on $M$ and by $L^2(M)$ the Hilbert space of real valued measurable square-integrable functions on $M$, endowed with the usual $L^2$ inner product.

Many formulas in the paper concern the case in which $M=S(V)$ is the unit sphere in a metric vector space $V$, that is, $S(V):=\{x\in V : \|x\|=1\}$ (see, e.g., \cref{t:wienerchaos}). In particular, $S^d=S(\R^d)$ denotes the standard round sphere. Finally, we define the constants
\be\label{eq:spherevolume}
s_{d}:=\vol{}(S^d)=\frac{2\pi^{\frac{d+1}{2}}}{\Gamma\tyu \frac{d+1}{2}\uyt},
\ee
for any integer $d$.

\section{Assumptions, Terminology, and Main Results} \label{sec:modelmain}

\subsection{A general model}
\label{sec:ourmodel}

The main results of our paper apply to a general family of random fields on Riemannian manifolds, which we will now describe in full detail. In the subsequent Section \ref{ss:randomwaves}, we will provide an exact connection between our framework and the general concept of {\it random wave}. 

\smallskip 

Let $(M,g)$ be a compact Riemannian manifold such that $\dim M=d$, and let $$Y_{m}\in \mC^\infty(M), \quad m\in \{1,\dots, N\},$$ be the elements of a real-valued orthonormal family of smooth functions in $L^2(M)$, that is, the functions $Y_1,...,Y_N$ satisfy the following
\be \label{eq:Yorthogonality}
\textbf{Orthonormality Assumption:}\quad  \langle Y_{ m},Y_{m'}\rangle_{L^2(M)}=\int_M Y_{m}(x)Y_{m'}(x)\dd x = \delta_{m,m'},
\ee 
where $\delta_{m,m'}$ is the Kronecker delta. We use the symbol $Y(x)\in \R^{N}$ to indicate the vector having $Y_{m}(x)$ ($m=1,...,N$) as coordinates, and we denote the standard scalar product on $\R^{N}$ by $\langle \cdot, \cdot \rangle$, whereas $\|\cdot\| = \langle \cdot, \cdot \rangle^{1/2} $ stands for the associated norm. We will work under the following isotropy-like assumption:
\be \label{eq:YlmNorm}
\textbf{Constant {Norm} Assumption: }\quad \|Y(x)\|=\sqrt{\frac{N}{\vol{}(M)}}, \ \forall x\in M.
\ee
For notational convenience, we also introduce the unit vector 
\begin{equation}\label{e:unitvector}
\y(x):=Y(x) \|Y(x)\|^{-1} {= Y(x) \,\sqrt{\frac{\vol{}(M)}{N}}} \in S^{N-1}.
\end{equation}

\medskip

Now let $\gamma=(\gamma^1,\dots,\gamma^{N})\sim \m N(0,\mathbbm{1}_{N})$ denote a standard Gaussian vector in $\R^{N}$. {In what follows, we write $f = \{f(x) : x\in M\}$ to indicate the smooth Gaussian field obtained as the unit-variance renormalization of the random mapping $x\mapsto \sum_{m=1}^N \gamma^m Y_m(x)$, $x\in M$, that is:
\be \label{eq:deff}
    f(\cdot): =\langle \gamma, Y(\cdot)\rangle \,\sqrt{\frac{\vol{}(M)}{N}}
=\langle \gamma, \y(\cdot)\rangle
\ \randin \ \mC^\infty(M),
\ee
where we have adopted the notation \eqref{e:randin}. 

\begin{definition}[\bf Uniform fields]\label{def:unif}
Let the above notation and assumptions prevail. The random field $\lebf := \{\lebf(x) : x\in M\}$ defined by the relation
\be \label{e:fcontilda}
\lebf(\cdot):=\frac{f(\cdot)}{\| f\|_{L^2(M)}}{=\left\langle \frac{\gamma}{\|\gamma\|}, Y(\cdot)\right\rangle} 
\ \randin \ \mC^\infty(M)
\ee
(recall \eqref{e:randin}), is called the \emph{uniform random field} associated with $f$.

\end{definition}
}

\smallskip 
\begin{remark}\label{rmk:CM}
From the point of view of Gaussian analysis, assumption \eqref{eq:Yorthogonality} is equivalent to requiring that the Cameron-Martin Hilbert space $\m H_f$ (see \cite{bogachev}) of $\sqrt{N/\vol{}(M)}f$ is isometric to a finite-dimensional Hilbert subspace of $L^2(M)$ through the canonical inclusions $\m H_f\subset \mC^\infty(M)\subset L^2(M)$. 
\end{remark}
\begin{remark}\label{rem:varVSvol}
Assumption \eqref{eq:YlmNorm} is equivalent to the requirement that the mapping $x\mapsto \|Y(x)\|$ is constant: as a matter of fact, by virtue of \cref{eq:Yorthogonality} one has that $\int_M\|Y(x)\|^2\dd x=N$, from which \eqref{eq:YlmNorm} necessarily follows. As already observed, condition \eqref{eq:YlmNorm} is equivalent to the requirement that the field $f(\cdot)$ has constant unit variance, that is
\be 
\E |f(x)|^2 =\frac{\E \|f\|^2_{L^2(M)}}{\vol{}(M)}=1, \quad \forall x\in M.
\ee
By combining this observation with the previous \cref{rmk:CM}, we can rephrase the settings described above, including the two assumptions in \cref{eq:Yorthogonality} and \cref{eq:YlmNorm}, in the following informal way: \emph{$f$ is a smooth Gaussian field on a Riemannian manifold $M$, with unit variance, having a finite-dimensional subspace of $L^2(M)$ as Cameron-Martin space.} 
\end{remark}
{
\begin{remark}\label{r:uniform} The following fact is well-known (and justifies our choice of the term ``uniform'' to describe the random field $\lebf$): the two $N$-dimensional vectors
\be\label{e:coeffs}
\sqrt{\frac{\vol{}(M)}{N}}\cdot \frac{\gamma}{\|f\|_{L^2(M)}} =\left(\frac{\gamma^1}{\|\gamma\|},...,\frac{\gamma^N}{\|\gamma\|}\right) \quad\mbox{and}\quad \left(\left(\frac{\gamma^1}{\|\gamma\|}\right)^2,...,\left(\frac{\gamma^N}{\|\gamma\|}\right)^2\right),
\ee
are distributed, respectively, as a uniform random vector on the sphere $S^{N-1}$, and as a Dirichlet random vector with parameters $\left( \frac12,...,\frac12\right)$. In particular, for all $i = 1,...,N$, the coefficient $\left(\gamma^i/ \|\gamma\|\right)^2$ follows a Beta distribution with parameters $\left(\frac12, \frac{N-1}{2}\right)$ and has therefore expectation equal to $\frac1N$. We also observe that, by rotational invariance and for all ${\bf v}\in S^{N-1}$,
\be\label{e:rotinv}
\left\langle\frac{\gamma}{\|\gamma\|}, {\bf v} \right\rangle \stackrel{d}{=} \frac{\gamma^1}{\|\gamma\|}.
\ee
\end{remark}

\smallskip 

\begin{remark}\label{rmk:normalization}
 Our normalization conventions for $\tilde{f}$ and $f$ differ. Indeed, applying \eqref{e:rotinv} to ${\bf v} = \y(x)$ (as defined in \eqref{e:unitvector}), one deduces immediately that $$\lebf(x)\stackrel{d}{=} \sqrt{\frac{N}{\vol{}(M) }} \cdot \frac{\gamma^1}{\|\gamma\|}, \quad x\in M.$$ 
 As a consequence,
\be\label{e:normtilde} 
 \E |\lebf(x)|^2=\frac{1}{\vol{}(M)}.
 \ee
\end{remark}

\subsection{A field guide to random waves}\label{ss:randomwaves} {This section aims to precise in what sense certain random fields within the framework of Section \ref{sec:ourmodel} are examples of \emph{random waves}. 

\smallskip

The term `random wave' was originally used to denote Berry's model of a random Laplace eigenfunction on the plane~\cite{Berry1977,Ber02}. 
Since then, it has been extended to various ensembles of random fields, including random Gaussian spherical harmonics (spherical random waves; see, for instance,~\cite{NazarovSodin2009,MaRossiJFA,Marinucci_Peccati_2011,MaWi2014,MarinucciRossiAHP2021,gts2025LerarioMarinucciRossiStecconi, Wig09}) and random Laplace eigenfunctions on tori (arithmetic random waves; see, e.g.,~\cite{ORW2008,RudnickWigmanTorus,MaPeRoWi2015,DNPR_EJP,Kri13}). 
S.~Zelditch later introduced the term \emph{Riemannian random wave} in~\cite{Zel09}, thereby extending the notion to general manifolds (see also~\cite{canzani_hanin_2015_immersions,canzani_hanin_2015_scalimkernel,canzani_hanin_2018_cinfscasymp,canzani_MRWsurvey,Gass2023,cgv2025StecconiTodino}). 
Such random fields fall within the broader class of \emph{random band-limited functions} studied in~\cite{wigman_survey_2024,Canzani_Sarnak_2019,SarnakWigman2019,wigEBNRF}. 
In all these settings, the underlying field is Gaussian, and a distinguished subclass consists of \emph{monochromatic} fields, which are approximate eigenfunctions of the Laplacian. 
Nevertheless, small inconsistencies persist across sources, as no official definition or unified terminology exists. 
For the purpose of this paper, we extend the use of the term \emph{random wave} beyond Gaussian fields to include uniform fields (see~\cref{def:unif}); from now on, we will adopt the following definition, which encompasses all the aforementioned cases.

\begin{definition}[Random wave]\label{d:whatisarandomwave}
Let $(M,g)$ be a compact Riemannian manifold. Let $0=\lambda_0< \lambda_1\le \lambda_2\le \dots \le \lambda_i \to+\infty $ be the frequencies (i.e. $-\lambda^2_i$ are eigenvalues) of the Laplace-Beltrami operator $\Delta\colon \mC^\infty(M)\to\mC^\infty(M)$, and let {$\f_1,\f_2,\dots,\f_i,\dots$} be a $L^2(M)$-orthonormal basis of real eigenfunctions such that $\Delta \f_{i}=-\lambda_i^2 \f_i$.
We say that a random field $f\randin \mC^\infty(M)$ is a \emph{random wave} if: $\P\{f= 0\}=0$ and there exists a bounded interval $I\subset \R$ such that, for every measurable $B\subset \mC^\infty(M)$, 
\be \label{e:whatisarandomwave}
\P\kop \frac{f}{\|f\|_{L^2(M)}} \in B\pok =\frac{1}{s_{N-1}} \vol{N-1} \tyu \kop \begin{pmatrix}
    a_1 \\ \vdots 
    \\
    a_N
\end{pmatrix} \in S^{N-1} \Bigg|\ \sum_{\{i\in \N:\lambda_i\in I\}} a_i \f_i \in B \pok\uyt,
\ee
where $N$ is the cardinality of the set $\{i\in \N:\lambda_i\in I\}$, {and the symbol $s_{N-1}$ is defined in \eqref{eq:spherevolume} for $d=N-1$}. Furthermore, we call such a $f$ a \emph{Gaussian random wave} if $f$ is Gaussian, or a \emph{uniform random wave} if $\|f\|_{L^2(M)}=1$, a.s.-$\mathbb{P}$. We call $f$ a \emph{purely monochromatic random wave} of frequency $\lambda>0$ if $f$ is a random wave and a random eigenfunction, that is: $\Delta f=-\lambda^2 f$, a.s.-$\mathbb{P}$ (i.e., $I$ in \eqref{e:whatisarandomwave} can be chosen to be the singleton $I=\{\lambda\}$). 
\end{definition}

A Gaussian random wave $f$, with interval window $I$, can be taken as the field $f$ of \cref{eq:deff} (from the setting of \cref{sec:ourmodel}) if and only if it has constant unit variance: 
\be\label{eq:unitvar_rw} 
\E f(x)^2 =1,\ \forall x\in M.
\ee 
The representation of such an $f$ in the form \eqref{eq:deff} is obtained by choosing $(Y_1,\dots,Y_N)=(\f_{i_1},\dots,\f_{i_N})$, where $i_1 \leq \cdots \leq i_N$ are the ordered elements of the set $T:=\{i\in \N\colon \lambda_i\in I\}$, and $N= |T|$. Then, the Orthonormality Assumption (\cref{eq:Yorthogonality}) is automatically satisfied by \cref{d:whatisarandomwave}, while the Constant Norm Assumption (\cref{eq:YlmNorm}) is equivalent to \cref{eq:unitvar_rw}, see \cref{rem:varVSvol}. In such a situation, it is sometimes convenient to use an alternative labeling of the Gaussian variables, so that the following holds:
\be \label{eq:relabel}
f=\sqrt{\frac{\vol{}(M)}{N}}\sum_{j=1}^N\gamma^j Y_j = \sqrt{\frac{\vol{}(M)}{N}}\sum_{\{i\in \N\colon \lambda_i\in I\}}\xi^i \f_i.
\ee
{Analogously to the previously introduced convention, this is achieved by setting $(\gamma^1,\dots,\gamma^N)=(\xi^{i_1},\dots,\xi^{i_N})$,} where $\xi^i:=\sqrt{\frac{N}{\vol{}(M)}}\int_M f(x)\f_i(x)dx$. In particular, $$\int_Mf(x)dx=\xi^0\sqrt{\frac{\vol{}(M)}{N}}\delta_0(I).$$

\medskip

\begin{remark}\label{r:afterRW}
\begin{enumerate}[(a)]
\item
Typically, the models introduced in Definition~\ref{d:whatisarandomwave} are considered for families of intervals~$I_\lambda$, with $\max\{x : x\in  I_\lambda\} \to +\infty$. 
In this setting, the term \emph{monochromatic} refers to the case $I_\lambda = [\lambda - o(\lambda),\, \lambda]$, or, more restrictively, to $I_\lambda = [\lambda - 1,\, \lambda]$; see e.g. \cite{Zel09,wigEBNRF,SarnakWigman2019,wigman_survey_2024,canzani_hanin_2015_immersions, canzani_hanin_2015_scalimkernel, canzani_hanin_2018_cinfscasymp, CH20}. 
For the purposes of the present paper, we do not need to enter into this level of detail.
\item
For $d \geq 2$, consider the case $M = S^d$ and, for $\ell = 1, 2, \ldots$, define
\be\label{e:sphericalharmonics}
N_\ell := \frac{2\ell + d - 1}{\ell} \binom{\ell + d - 2}{\ell - 1},
\qquad 
Y_{\ell,\bullet} := \{ Y_{\ell,m} : m = 1, \ldots, N_\ell \},
\ee
where $Y_{\ell,\bullet}$ denotes the class of real-valued \emph{spherical harmonics} of degree~$\ell$, as defined for instance in~\cite[Chapter~9]{SpecialFunctions} or~\cite{MaRossiJFA,Marinucci_Peccati_2011}. 
As explained in these references, for every $\ell \geq 1$ the set $Y_{\ell,\bullet}$ forms an orthonormal basis of the eigenspace of the spherical Laplace-Beltrami operator $\Delta_{S^d}$ (acting on $\mathcal{C}^\infty(S^d)$) associated with the eigenvalue $-\ell(\ell + d - 1)$. In the following, we will denote such space as $\m H_{\ell,d+1}$, see \cref{def:sph_harm_HqN}.
By orthogonal invariance, one consequently deduces that $Y_{\ell,\bullet}$, if interpreted as a vector valued function $Y_{\ell,\bullet}\colon M\to \R^{N_\ell}$, satisfies both Assumptions~\ref{eq:YlmNorm} and~\ref{e:unitvector}. The corresponding random field~$f$ defined in~\eqref{eq:deff} coincides with the model of \emph{random spherical harmonics} considered, for instance, in~\cite{Marinucci_Peccati_2011,MaWi2014,MarinucciRossiAHP2021,Wig09,MaRossiJFA, MRW20, CammarotaMarinucciWigmanJGA2016, MaWi2011excarea, NazarovSodin2009}. 
In this setting, the model corresponding to the uniform random field~$\tilde{f}$ in \eqref{e:fcontilda} will be referred to as the model of \emph{uniform random spherical harmonics}; as already discussed in the Introduction, uniform random spherical harmonics have been studied in \cite{FengAdler19, FengXuAdler18, FengYaoAdler2025}.  In the terminology of Definition~\ref{d:whatisarandomwave}, the fields~$f$ and~$\tilde{f}$ are, respectively, a purely monochromatic Gaussian random wave, and a purely monochromatic uniform random wave. 
As already pointed out, understanding variance estimates and cancellation phenomena for geometric functionals of uniform random waves on the sphere constitutes one of the main motivations for the present work.

\item For $d\geq 2$, the model of \emph{arithmetic random waves} studied e.g. in \cite{rozensheinIMRN,Cam19, ORW2008, RudnickWigmanTorus, Kri13, MaPeRoWi2015, DNPR23, benatarMW, sartoriIMRN} can be obtained by considering the case $M = \mathbb{T}^d := \R^d/\mathbb{Z}^d$, and by selecting the orthonormal system
\be\label{e:aw}
Y_{(n)} := \{\sqrt{2}\cos( 2\pi\langle \lambda, \cdot \rangle_{\R^d}), \sqrt{2}\sin( 2\pi\langle \lambda, \cdot  \rangle_{\R^d}) : \{\lambda, -\lambda\} \in \Lambda_n/\pm \},    
\ee
where the integer $n\geq 1$ is such that $4\pi^2 n$ is an eigenvalue of the Laplacian $\Delta$ on $\mathbb{T}^d$, and $\Lambda_n/\pm$ indicates the quotient $\Lambda_n/ \mathbb{Z}_2$, where $\Lambda_n$ is the class of all $\lambda = (\lambda_1,...,\lambda_d)\in \mathbb{Z}^d$ such that $\|\lambda\|^2_{\R^d} = n$. Similarly to Point (b), the corresponding random fields $f$ and $\tilde{f}$, defined respectively in \eqref{eq:deff} and in \eqref{e:fcontilda}, are instances of purely monochromatic Gaussian random waves and purely monochromatic uniform random waves.

\item \label{itm:HomogeneouSpaz} The framework introduced in Section~\ref{sec:ourmodel} includes random waves, in the sense of Definition~\ref{d:whatisarandomwave}, defined on manifolds more general than the sphere or the torus (see Points~(b) and~(c)) --- although the Constant Norm Assumption~\eqref{eq:YlmNorm} imposes a significant geometric constraint, introducing a form of rigidity.\footnote{Most manifolds do not admit any transitive group action, let alone one by isometries. In fact, a generic Riemannian metric on a smooth manifold has only the identity as an isometry.} 
Manifolds for which all random waves naturally fit within our framework are e.g. the \emph{homogeneous} ones, that is, those for which for any pair of points $x,y \in M$, there exists an isometry $\iota \colon M \to M$ such that $\iota(x)=y$. Examples include: the (already mentioned) flat torus~$\mathbb{T}^d$ and any other flat torus; real, complex, and quaternionic projective spaces; higher Grassmannians; compact Lie groups; homogeneous spaces. 
A notable class that falls outside this setting is that of hyperbolic surfaces: since their isometry groups are finite, they are not homogeneous, and indeed Condition~\ref{eq:YlmNorm} fails in this case.

\item It should be noted that the above definition is intrinsically finite-dimensional, and hence it cannot include Berry's random waves on Euclidean spaces. 

\end{enumerate}
\end{remark}

\subsection{Wiener chaos}\label{sec:Wichaos}

The next definition is standard and will be used in forthcoming Definition \ref{d:berrycancellation} as a tool for encoding the notion of {\it cancellation} evoked in the Introduction. See e.g. \cite[Chapters 1 and 2]{nourdinpeccatibook} for more details.

\begin{definition}[\bf Hermite polynomials \& Wiener chaos]\label{d:wiener} For every integer $k\geq 0$, the $k$th (probabilistic) \emph{Hermite polynomial} $H_k : x\mapsto H_k(x)$ is defined according to the following recurrence relation:
$$
H_0\equiv 1, \quad H_{k}(x) := xH_{k-1}(x) - H'_{k-1}(x),\quad x\in \R, \,\,\, k\geq 1,
$$
so that $H_1(x) = x$, $H_2(x) = x^2-1$, $H_3(x) = x^3-3x$, $H_4(x) = x^4-6x^2+3$, and so on. Writing $\varphi(x):= (2\pi)^{-1/2} \exp(-x^2/2)$, one has that the family $\{H_k/\sqrt{k!} : k\geq 0\}$ is an orthonormal basis for the real Hilbert space $L^2(\R, \, \varphi(x) {\rm d} x)$, yielding the well-known relation
\be\label{e:orthohermite}
\int_\R H_k(x) H_{\ell}(x) \phi(x)\,{\rm d}x = k!\, \delta_{k,\ell}. 
\ee
Now fix a finite or countable index set $J$, and consider a family $\gamma = \{\gamma^j : j \in J\} $ of independent $\mathscr{N}(0,1)$ random variables. We write $L^2(\gamma)$ to indicate the Hilbert space composed of all (equivalence classes of) $\sigma(\gamma)$-measurable random variables $F = F(\gamma)$ such that $\E[F^2]<\infty$. For every $q\geq 0$, we will write $C_q$ to denote the $q$th \emph{Wiener chaos} associated with $\gamma$, where the spaces $C_q$ are defined as follows: $C_0=\mathbb{R}$ and, for all $q\geq 1$, $C_q$ is the subspace of $L^2(\gamma)$ generated by the random variables\footnote{Observe that, trivially, in the infinite product \eqref{e:bighermite} there can be at most $q$ indices $k_i\neq 0$. }
\be\label{e:bighermite}
\prod_{i\geq 1} H_{k_i}(\gamma^i), \quad \mbox{verifying}\quad \sum_{i\geq 1} k_i = q. 
\ee
\end{definition}

\medskip 

It is easily checked that, if $|J| = N<+\infty$, then $C_q$ is a finite-dimensional vector space, with dimension equal to the total number of \emph{weak compositions} of length $N$ of the integer $q$ (see e.g. \cite[p. 8]{PecTaqqBook}); {moreover, in this case, we adopt the \emph{multi-index notation}: we define $H_k(\gamma):= \prod_{i=1}^N H_{k_i}(\gamma^i)$, where $k\in \N^N$ is such that $|k|:=k_1+\ldots+k_N=q$, and $k!:= \prod_{i=1}^N k_i!$.} In general, by virtue of \eqref{e:orthohermite}, one has that $C_k \perp C_\ell $ for $k\neq \ell$, where the orthogonality relation is in the sense of $L^2(\gamma)$. For every $F\in L^2(\gamma)$, we will often use the compact notation
\be\label{e:projection}
F[q] := {\rm proj}\,(F \, |\, C_q), \quad q\geq 0,
\ee
where ``${\rm proj}$'' stands for an orthogonal projection; in particular, $F[0] = \E[F]$. The following fundamental relation is a direct consequence of the fact that the sequence $\{H_k/\sqrt{k!}\}$ is an orthonormal basis of $L^2(\R, \, \varphi(x) {\rm d}x)$, and is often referred to as the \emph{Wiener-It\^o chaos expansion} of $L^2(\gamma)$: one has that $L^2(\gamma) = \bigoplus_{q\geq 0}C_q$, that is, for every $F\in L^2(\gamma)$,
\be\label{e:wienerexp}
F = \sum_{q=0}^\infty F[q],
\ee
where the (orthogonal) series on the right-hand side converges in $L^2(\gamma)$. 

\medskip
 
\begin{remark}\label{r:doh} \begin{enumerate}[(a)]
\item The framework outlined in Section \ref{sec:ourmodel} corresponds to the case $J = \{1,...,N\}$, and $\gamma = (\gamma^1,...,\gamma^N)\sim\mathcal{N}(0, \mathbb{I}_N)$. In particular, any square-integrable random variable $F = F(f)$ depending on the random field $f$ defined in \eqref{eq:deff} (and therefore any square-integrable random variable depending on the uniform field $\tilde{f}$ introduced in \eqref{e:fcontilda}) is an element of the space $L^2(\gamma)$ and consequently admits the Wiener-It\^o decomposition \eqref{e:wienerexp}.
\item We recall that, for all $q\geq 1$, a {\it symmetric tensor} $K$ on $[N]^q$ is a symmetric mapping
$$
K : [N]^q\to \R : (i_1,...,i_q) \mapsto K(i_1,...,i_q)
$$
(in what follows, we write $\mathbb{T}(q,N)$ to indicate the collection of all symmetric tensors on $[N]^q$, noting that $\mathbb{T}(1,N)=\R^N$). If $\gamma = (\gamma^1,..., \gamma^N)\sim\mathcal{N}(0, \mathbb{I}_N)$, then it is a well-known fact that, for all $q\geq 1$, a random variable $X$ belongs to the $q$th Wiener chaos $C_q$ if and only if there exists a (unique) tensor $K_X \in \mathbb{T}(q,N)$ such that
\begin{equation}\label{e:chaosistensor}
X = \sum_{(i_1,..,i_q)\in[N]^q } K_X(i_1,...,i_q) \,\,: \gamma^{i_1}\cdots \gamma^{i_q} :\,\,,
\end{equation}
    where $: \gamma^{i_1}\cdots \gamma^{i_q} :$ indicates a {\it Wick product}, as defined in Appendix \ref{ss:AAWick}. See e.g. \cite[Chapters 2-3]{JansonBook} for a full discussion.  
    
    \item Let again $\gamma = (\gamma^1,..., \gamma^N)\sim\mathcal{N}(0, \mathbb{I}_N)$. The following relation is well-known: for all $q\geq 1$ and all $X,Y\in C_q$, noting $K_X, K_Y\in \mathbb{T}(q,N)$ the tensors associated with $X,Y$ via \eqref{e:chaosistensor}, one has that
\begin{equation}\label{e:innertenso}
\mathbb{E}[XY] = q! \sum_{(i_1,..,i_q)\in[N]^q } K_X(i_1,...,i_q) K_Y(i_1,...,i_q) := q! \, \langle K_X, K_Y\rangle,
\end{equation}
yielding that the mapping $C_q\to  \mathbb{T}(q,N) : X\mapsto K_X $ is indeed an isometry. For the sake of completeness, a proof of \eqref{e:innertenso} is sketched in Appendix \ref{ss:AAWick}.
\end{enumerate}
\end{remark}

\subsection{Wiener chaos and spherical harmonics under $0$-homogeneity} \label{ss:wcsh}

We now establish what appears to be a new connection between Wiener chaos (as defined in the previous section) and spherical harmonic expansions for
$0$-homogeneous functionals. To this end, for the rest of this section we fix an integer $N\geq 2$, write $\gamma\sim \mathcal{N}(0, \mathbb{I}_N)$, and introduce the shorthand notation 
\begin{equation}\label{e:unithing}
U = (U^1,...,U^N) := \frac{\gamma}{\|\gamma\|}. 
\end{equation}
As already observed in Remark \ref{r:uniform}, the vector $U$ has a uniform distribution on $S^{N-1}$. For $q\geq 1$, we say that a tensor $K\in \mathbb{T}(q,N)$ is {\it traceless} if
\begin{equation}\label{e:traceless}
\sum_{j=1}^N K(j,j, i_3,...,i_{q}) = 0, \quad \forall i_3,...,i_{q}\in [N]
\end{equation}
(note that, since $K$ is symmetric, the position of the two matching indices $j,j$ in the argument of $K$ is immaterial). We use the symbol $\mathbb{T}_0(q,N)$ to denote the class of traceless symmetric tensors, and note that $\mathbb{T}_0(1,N) = \mathbb{T}(1,N) = \R^N$. 

\begin{remark}\label{r:projectiong} Since $\mathbb{T}_0(q,N)$ is a linear subspace of $\mathbb{T}(q,N)$, one has the following orthogonal decomposition: for all $K\in \mathbb{T}(q,N)$,
\begin{equation}\label{e:tracelessD}
K = {\rm proj}(K\, |\, \mathbb{T}_0(q,N)) + {\rm proj}(K\, |\, \mathbb{T}_0(q,N)^{\perp}) := K^{\rm (TL)} + K^{\rm (Tr)},
\end{equation}
where ``${\rm proj}$'' stands for an orthogonal projection with respect to the inner product $\langle \cdot, \cdot\cdot \rangle$ defined in \eqref{e:innertenso}. For instance, in the case $q=2$, one has that $$K^{\rm (TL)} = K - \frac{{\rm Tr}\,K}{N} {\mathbb{I}_N}.$$
\end{remark}
\begin{definition}[See e.g. Chapter IV in \cite{SteinWeissBook}]\label{def:sph_harm_HqN} {Consistently with the notation introduced in \cref{r:afterRW}.(b) above,} for $q\geq 0$ we write $\mathcal{H}_{q,N}$ to indicate the (real) vector space generated by degree-$q$ spherical harmonics on the sphere $S^{N-1}$, that is: $\mathcal{H}_{0,N} = \R$ and, for $q\geq 1$, $ h \in \mathcal{H}_{q,N}$ if and only if $h$ is the restriction to $S^{N-1}$ of a homogeneous harmonic polynomial on $\R^N$. {Note that the space $\m H_{q,d+1}$ is generated by the set $Y_{\ell,\bullet}=\kop Y_{\ell,m}:m=1,\dots,N_\ell\pok$ of standard spherical harmonics, as defined in the already mentioned \cref{r:afterRW}.(b).}
\end{definition}
\begin{remark}
We recall that a homogeneous polynomial $p$ of degree $q\geq 1$ on $\mathbb{R}^N$ is harmonic (that is, $\Delta p=0$) if and only if
$$
p(x_1,...,x_N) = \sum_{(i_1,...,i_q)\in [N]^q} K(i_1,...,i_q) \, x_{i_1}\cdots x_{i_q},
$$
for some traceless $K\in \mathbb{T}_0(q,N)$. More precisely, a symmetric tensor can be uniquely identified with its induced homogeneous polynomial (i.e. the evaluation of the tensor in the direction of a single repeated vector). By this identification, the trace of a tensor is, up to an explicit constant, the Laplacian of this polynomial. A traceless $q$-symmetric tensor is then an harmonic $q$-homogeneous polynomial.
\end{remark}

\begin{definition}[\bf $0$-homogeneity]\label{d:0hom}
A mapping $E : \R^N\to \R$ is said to be (positively) $0$-{\it homogeneous} if $E(x) = E(t x)$, for every $t>0$, so that, in particular,
$$
E(\gamma) = E(U), \quad \mbox{a.s.-}\P.
$$
\end{definition}

\medskip

Now let $E : \R^N\to \R$ be a $0$-homogeneous mapping such that $E(\gamma)$ is square-integrable. Using the notation \eqref{e:wienerexp}, one has that
\begin{equation}\label{e:2projections}
E(\gamma) = E(U) = \sum_{q=0}^{\infty} E(\gamma)[q] = \sum_{q = 0}^\infty {\rm proj} \{ E \, |\, \mathcal{H}_{q,N} \}(U),
\end{equation}
where both series converge in $L^2(\mathbb{P})$, and 
$$
{\rm proj} \{ E \, |\, \mathcal{H}_{q,N} \} : S^{N-1}\to \R : v\mapsto {\rm proj} \{ E \, |\, \mathcal{H}_{q,N} \}(v)
$$
indicates the projection of the mapping $v\mapsto E(v) \in L^2(S^{N-1})$ onto $\mathcal{H}_{q,N}$. The next statement (which, to the best of our knowledge, is new) completely clarifies the connection between the summands appearing in the two infinite sums displayed in \eqref{e:2projections}.

\begin{theorem}\label{t:nonmale} Let $E : \R^N\to \R$ be a $0$-{homogeneous mapping} such that $E(\gamma)$ is square-integrable. For every $q\geq 1$, let $K_q \in \mathbb{T}(q,N)$ be such that
\begin{equation}\label{e:zorro}
E(\gamma)[q] = \sum_{(i_1,...,i_q)\in [N]^q} K_q(i_1,...,i_q) \, :\gamma^{i_1}\cdots\gamma^{i_q}:.
\end{equation}
Then, for every $q\geq 1$ one has that 
\begin{equation}\label{e:notnaiveatall}
{\rm proj} \{ E \, |\, \mathcal{H}_{q,N} \}(U) = c_{q,N} \, \sum_{(i_1,...,i_q)\in [N]^q} K^{\rm (TL)}_q(i_1,...,i_q) \, U^{i_1}\cdots U^{i_q},
\end{equation}
where we have adopted the notation \eqref{e:tracelessD}, and
$$
c_{q,N} := \frac{\mathbb{E}[\|\gamma\|^{2q}]}{\mathbb{E}[\|\gamma\|^{q}]} =\frac{2^{q/2} \Gamma(\frac12(N+2q))}{\Gamma(\frac12(N+q))}.
$$
Moreover, one has always that ${\rm Tr}\, K_2 = 0$, yielding  $K_2 = K_2^{\rm (TL)}$ and 
\begin{equation}\label{e:maguardachebello}
{\rm proj} \{ E \, |\, \mathcal{H}_{2,N} \}(U) = (N+2)\, U^T K_2 U.
\end{equation}
\end{theorem}

\medskip 

\begin{remark}Let $E(\cdot)$ be as in Theorem \ref{t:nonmale}. Elementary computations yield that, for all $q\geq 1$,
$$
{\bf Var}({\rm proj} \{ E \, |\, \mathcal{H}_{q,N} \}(U)) = q! \, \hat{c}_{q,N} \|K_q^{\rm (TL)}\|^2,
$$
where 
$$
\hat{c}_{q,N} = \frac{\Gamma(\frac12 N) \Gamma(\frac12 (N+2q))}{\Gamma(\frac12 (N+q))^2}.
$$
It follows that
$$
{\bf Var}(E(\gamma)[q]) = q! \|K_q^{\rm (TL)}\|^2 + q!\|K_q^{\rm (Tr)}\|^2 = \frac{1}{\hat{c}_{q,N}}{\bf Var}({\rm proj} \{ E \, |\, \mathcal{H}_{q,N} \}(U))+ q!\|K_q^{\rm (Tr)}\|^2. 
$$
In particular, the last relation yields the bound
$$
{\bf Var}(E(\gamma)[q])\geq \frac{1}{\hat{c}_{q,N}}{\bf Var}({\rm proj} \{ E \, |\, \mathcal{H}_{q,N} \}(U)),
$$
as well as the implication
\begin{equation}\label{e:implica}
E(\gamma)[q] = 0 \quad \Longrightarrow \quad {\rm proj} \{ E \, |\, \mathcal{H}_{q,N} \}(U)=0.
\end{equation}
Except for the cases $q=1,2$, relation \eqref{e:implica} cannot be reversed.\footnote{To see this, let $\gamma = (\gamma^1, \gamma^2)$ be a standard Gaussian vector, and define the $0$-homogeneous random variable
$$
X =\frac{\gamma^1}{\|\gamma\|} + \frac{(\gamma^1)^2 - (\gamma^2)^2}{\|\gamma\|^2} = U^1+ (U^1)^2 - (U^2)^2. 
$$
Then, $X = h(U^1,U^2)$, where $h\in \mathcal{H}_{1,2}\oplus \mathcal{H}_{2,2}$ (that is, $h$ is the sum of a spherical harmonic of degree 1 and of a spherical harmonic of degree 2), but $X[q] \neq 0$ for every $q\geq 3$. One can check that $X[2] = \frac14\{(\gamma^1)^2 - (\gamma^2)^2\}$, consistently with \eqref{e:maguardachebello} for $N=2$.}
\end{remark}

\medskip

The proof of Theorem \ref{t:nonmale} --- which is based on two results of independent interest --- is provided in Section \ref{A:SHproofs}.  In the next section, we will use chaotic expansions to precisely capture the phenomena that are the main focus of our work.}

\subsection{Cancellations}\label{ss:cancellations}

\begin{definition}[\bfseries Cancellation phenomena]\label{d:berrycancellation}
Let the above notation and conventions prevail (with $\gamma = \{\gamma^j : j\in J\}$ as in Definition \ref{d:wiener}) and let  
\begin{equation}\label{e:thegfield}
G \, \randin \, \mathcal{C}^\infty(M)
\end{equation}
be a random field. Let $\mathbb{U}$ be an arbitrary parameter space, and consider a collection $\Phi := \{\varphi_u : u\in \mathbb{U}\}$ of real-valued functionals
\[
\varphi_u : \mathcal{C}^\infty(M) \to \R, \quad u\in \mathbb{U},
\]
such that $\varphi_u(G) \in L^2(\gamma)$ for all $u\in \mathbb{U}$.

\begin{itemize}
\item[(i)] 
Given an integer $r\geq 1$ and $\mathbb{U}_0\subseteq \mathbb{U}$, we say that $G$ \emph{induces a sharp $r$th-order (chaotic) cancellation on} $(\Phi, \mathbb{U}_0)$ if
\begin{equation}\label{e:kcancellation}
{\varphi_u(G)[r] = 0, \quad \text{a.s.-}\P \iff u\in \mathbb{U}_0,}
\end{equation}
where we have used the notation in \eqref{e:projection}.  When $\mathbb{U}_0 = \mathbb{U}$, we will say that $G$ \emph{induces a $r$th-order maximal cancellation on} $(\Phi, \mathbb{U})$. 
\item[(ii)] 
For $i=0,1$, let $
G_i \, \randin \,  \mathcal{C}^\infty(M) $ be random fields such that $\varphi_u(G_i)  \in L^2(\gamma)$ for all $u\in \mathbb{U}$.  We say that $G_1$ \emph{induces a stronger $r$th-order cancellation} on $\Phi$ than $G_0$ if, for $i=0,1$, the field $G_i$ induces a sharp $r$th-order cancellation on $(\Phi, \mathbb{U}_i)$, where the subsets $\mathbb{U}_0,\mathbb{U}_1\subseteq \mathbb{U}$ are such that $\mathbb{U}_0\subsetneq \mathbb{U}_1$.

\item[(iii)] 
 Let $\Delta$ denote the Laplace–Beltrami operator associated with the underlying compact Riemannian manifold $(M,g)$. We say that $\Phi$ {\it yields a cancellation of Berry's type}\footnote{This nomenclature is a homage to the seminal work~\cite{Ber02}, where second-order cancellation phenomena were first identified for level-set functionals of planar random waves.} if the following implication holds: if $G$ is a Gaussian (resp. uniform) random wave in the sense of Definition \ref{d:whatisarandomwave} and $G$ verifies \eqref{e:kcancellation} for some $\mathbb{U}_0\subseteq \mathbb{U}$, then necessarily $G$ is a purely monochromatic Gaussian (resp. uniform) random wave.
\end{itemize}
\end{definition}

\begin{example}\label{e:stecconitodinoarethebest} Let $(M,g)$ be a compact $d$-dimensional Riemannian manifold on which random waves (according to Definition \ref{d:whatisarandomwave}) are {\it homothetic} in the sense of \cite[Definition 1.11]{cgv2025StecconiTodino}; recall that this framework includes spheres, tori, as well as any manifold that admits an isometric action whose differential is transitive on the tangent bundle. Then, \cite[Corollary 2.1]{cgv2025StecconiTodino} (see also \cite[Corollary 3.10]{cgv2025StecconiTodino}) establishes that \emph{Gaussian random waves on $(M,g)$ induce a sharp 1st and 2nd-order cancellation of Berry's type on $(\Phi,\{0\})$}, where $\Phi := \{ \f_u : u\in \R\}$ is the collection of all level volume functionals $\f_u(G)=\vol{d-1}(G-u)$. Such a result establishes the sharpness of the classical cancellation phenomena on spheres (see e.g. \cite{Wig09, MaPeRoWi2015, MRW20}) or tori (see e.g. \cite{Cam19, Kri13, MaPeRoWi2015, Not21, DNPR_EJP}), and mirrors analogous cancellations in Euclidean waves --- see, for instance, \cite{Ber02, NourdinPeccatiRossi2019, Dalmao2021,gass2025spectralcriteriaasymptoticslocal, wigman_survey_2024} and the references therein. A new maximal 2nd order cancellation of Berry's type for uniform random waves --- applying to a large class $\Phi$ of geometric functionals admitting an integral representation --- will be established in forthcoming Theorem \ref{thm:louis2}.
\end{example}

The upcoming sections detail how our results uncover several general mechanisms of sharp cancellation characteristic of uniform random fields.

\subsection{Main results, I: local and nonlocal functionals of uniform random spherical harmonics}\label{subsec:mainI}

Our first main result implies that \emph{uniform random spherical harmonics} (as defined in Remark \ref{r:afterRW}-(b)) exhibit a form of \emph{{ 1st and 2nd order} maximal sharp cancellation} on a large class of level-set functionals --- in decisive contrast with the cancellations observed in Gaussian random waves (see Example \ref{e:stecconitodinoarethebest}), which occur only at level zero. The general mechanism underlying this phenomenon is described in Theorem~\ref{thm:Main1}, while its explicit application to uniform random spherical harmonics is presented in Corollaries~\ref{c:splendid0} and~\ref{c:splendid1} below.

\vspace{0.3em}

The crux of our argument relies on the fact that spherical harmonics of any fixed degree on $S^d$ provide an irreducible representation of the compact rotation group~$SO(d+1)$. This property, combined with the natural homothety (that is, invariance with respect to any scaling factor) associated with geometric functionals of uniform fields, is sufficient to yield the maximal cancellation described above.

 \begin{mainthm}\label{thm:Main1} 

{\it Let $G$ be a compact group, fix $N\geq 2$, and let $$L\colon G\to SO(N): R\mapsto L(R)=:R^*$$ be a real representation of $G$ having the form 
\begin{equation}\label{e:reduci}
L(R) = \bigoplus_{\ell=1}^M L_\ell(R),
\end{equation}
where $1\leq M\leq N$, each $L_\ell : G\rightarrow SO(N_\ell)$ is irreducible and, for $\ell\neq k$, $L_\ell$ and $L_k$ are not equivalent (see {\rm \cite[p. 96]{FarautBook}} for definitions, and remark the constraint $N = \sum_{\ell=1}^M N_\ell$). Let $\gamma \sim \mathcal{N}(0,\mathbbm{1}_N)$, let $E\colon \R^N\to \R$ be a $0$-homogeneous mapping (in the sense of Definition \ref{d:0hom}) such that
\be \label{e:guapo}
    E(R^*\gamma)=E(\gamma),
    \ee 
for all $R\in G$. Assume that $E(\gamma)$ is square-integrable; 
{then, writing $\gamma =(\Gamma_1,\dots,\Gamma_M)$, with
$$
\Gamma_\ell := (\gamma_{ N_1+\cdots+N_{\ell-1} +1}, ..., \gamma_{ N_1+\cdots+N_{\ell} } ), \quad \ell=1,..., M,
$$
there exist $k_1,\dots,k_M\in \R$ such that
\be 
E(\gamma)[2]=\sum_{\ell=1}^M k_\ell \|\Gamma_\ell\|^2,\quad \text{and}\quad \sum_{\ell=1}^Mk_\ell N_\ell=0.
\ee
}}

\end{mainthm}

The conclusion of the above theorem implies in particular that, if we write $K_2$ to denote the symmetric $N\times N$ matrix defined in \eqref{e:zorro} for $q=2$, then $K_2$ is a diagonal matrix such that ${\rm Tr}\, K_2 = 0$ and whose diagonal has at most $M$ distinct entries.
In the rest of the paper, we will make use of the following corollary of \cref{thm:Main1} above, obtained by specializing the statement to the case $M=1$. 
\begin{corollary}\label{c:ofmain1}
Let $G$ be a compact group, fix $N\geq 2$, and let $L\colon G\to SO(N): R\mapsto L(R)=:R^*$ be an irreducible real representation of $G$. Let $\gamma \sim \mathcal{N}(0,\mathbbm{1}_N)$ and let $E\colon \R^N\to \R$ be a mapping such that $E(\gamma)$ is square-integrable and
 \be \label{e:guapino}
     E(R^*\gamma)=E(\gamma) \quad \text{and} \quad E(t\gamma)=E(\gamma),
\ee 
for all $R\in G$ and all $t\in\mathbb R$. Then, $E(\gamma)[1] = E(\gamma)[2] = 0$.
\end{corollary}

The next statement identifies two fundamental consequences of Theorem \ref{thm:Main1}, in the case of random spherical harmonics. {Indeed, it is well known that the canonical action of the orthogonal group $O(N)$ on the space $\m H_{q,N}$ of all spherical harmonics of a given degree $q\in \N$, induces an irreducible representation, to which we can apply \cref{c:ofmain1}, see \cref{sec:main1} for details.}
{We use the notation $\{\mathcal{L}_a : a=0,1,...,d\}$ and $\{b_a : a=0,1,...,d\}$ to denote, respectively, Lipschitz-Killing curvatures (defined as in \cite[Definition 10.7.2]{AdlerTaylor}) and Betti numbers (meant in the usual sense of algebraic topology, for which we refer to \cite[Chapter 4]{spanier1995algebraic}), see also \cref{ex:typeW}. }

\begin{corollary}\label{c:splendid0} For $d\geq 2$ let $\tilde{f} = \{\tilde{f}(x) : x\in S^d\}$ be the model of uniform random spherical harmonics of degree $\ell\geq 1$, as defined in Remark {\rm \ref{r:afterRW}-(b)} (see also \eqref{e:fcontilda}). Let $j^k\lebf \colon S^d\to J^k(S^{d})$ be the jet prolongation of $\lebf$ and let $\Sigma\subset J^k(S^{d},\R)$ be isotropic and semialgebraic (see \cref{def:isosing} and \cref{def:semialgsing}); then, $\mathcal{L}_a\tyu j^k\lebf \in \Sigma \uyt, b_a\tyu j^k\lebf \in \Sigma \uyt$ are square integrable and
\be 
\mathcal{L}_a\tyu j^k\lebf \in \Sigma \uyt[q] = b_a\tyu j^k\lebf \in \Sigma \uyt[q]=0, \quad q=1,2,
        \ee
        for all $a\in\{0,\dots, d\}$.
\end{corollary}

The following corollary establishes the announced form of \emph{maximal cancellation} for geometric functionals of uniform random spherical harmonics, encompassing in particular instances of \emph{non-local functionals}---such as Betti numbers---that typically fall outside the scope of Wiener chaos-based techniques (see, however, \cite{mcauley2025limittheoremsnumbersign}). As already noted, this result should be contrasted with the content of Example~\ref{e:stecconitodinoarethebest}, showing that for Gaussian random waves on general manifolds (a setting that includes Gaussian spherical harmonics) the second-order cancellation of level-set volumes is sharp but \emph{non-maximal}, being in fact concentrated only at the nodal level.

\begin{corollary}\label{c:splendid1}
    For $d\geq 2$ let $\tilde{f} = \{\tilde{f}(x) : x\in S^d\}$ be as in Corollary \ref{c:splendid0}. Then, $\tilde{f}$ induces a maximal {1st and 2nd order} cancellation on $(\Phi, \R)$ (according to Definition \ref{d:berrycancellation}-{\rm (i)}), where 
    $$
    \Phi = \{\varphi_u : u\in \R\}
    $$
is a collection of functionals belonging to any of the following classes:
    \begin{enumerate}[$(i.)$]
        \item Excursion volumes: $ \varphi_u(\tilde{f}) = \vol{d}(\lebf\ge u)$, $u\in \R$;
        \item Level volumes: $\varphi_u(\tilde{f}) =\vol{d-1}(\lebf = u)$, $u\in \R$;
        \item Critical points: $\varphi_u(\tilde{f}) = \#\{d\lebf=0, \lebf\ge u\}$,  $u\in \R$;
        \item Lipschitz-Killing curvatures: for fixed $a=0,1,...,d$, $\varphi_u(\tilde{f}) =\mathcal{L}_a\tyu \lebf \ge u \uyt$, or  $\varphi_u(\tilde{f}) = \mathcal{L}_a\tyu \lebf = u \uyt$,  $u\in \R $;
        \item Betti numbers: for fixed $a=0,1,...,d$, $\varphi_u(\tilde{f}) =b_a\tyu \lebf \ge u \uyt$, or  $\varphi_u(\tilde{f}) = b_a\tyu \lebf = u \uyt$, $u\in \R $.
    \end{enumerate}
\end{corollary}

\begin{remark}{\rm In the case of critical points (see Point ({\it iii.}) of the previous statement), one can replace set $\R$ with $\R\cup \{-\infty\}$. 
}
\end{remark}

\smallskip

\begin{remark}[On irreducibility]\label{ex:contrexirred}{\rm In Theorem \ref{thm:Main1}, the assumption that $L$ is irreducible cannot be easily removed if one wants to conclude that $E(\gamma)[2] = 0$. To see this, we let the notation of Theorem \ref{thm:Main1} prevail and, for some $M\geq 2$, we consider $M$ irreducible real representations $L_\ell : G\mapsto SO(N_i)$ of $G$, where $N_\ell\geq 2$, $\ell=1,...,M$. Writing $N = \sum_{\ell=1}^M N_\ell$, one has therefore that 
$$
R\mapsto \bigoplus_{\ell=1}^M L_\ell(R) := L(R),
$$
is a reducible real representation of $G$ over $\R ^N$. Now let $\gamma \sim \mathcal{N}(0,\mathbbm{1}_N)$ as before, and consider a vector $(b_1,...,b_M)\in \R^M\backslash\{0 \}$ such that $\sum_{\ell =1}^M b_\ell N_\ell = 0$. Setting $N_0 := 0$, and defining
$$
\Gamma_\ell := (\gamma_{ N_1+\cdots+N_{\ell-1} +1}, ..., \gamma_{ N_1+\cdots+N_{\ell} } ), \quad \ell=1,..., M,
$$
one has therefore that
$$
E(\gamma) := \frac{\sum_{\ell=1}^M b_\ell \|\Gamma_\ell \|^2}{\|\gamma\|^2}
$$
is a bounded $0$-homogeneous random variable verifying \eqref{e:guapo}. However, since $X:= \sum_{\ell=1}^M b_\ell \|\Gamma_\ell \|^2$ is a non-zero element of the second Wiener chaos associated with $\gamma$ and 
$$
\mathbb{E}[XE(\gamma)[2]]= \mathbb{E}[XE(\gamma)] = \mathbb{E}[X^2/\|\gamma\|^2]>0,
$$
one has that $\E(\gamma)[2]\neq 0$, that is, the conclusion of Theorem \ref{thm:Main1} is not valid in this case.
}
\end{remark}

In the next subsection, the conclusion of Corollaries \ref{c:splendid0} and \ref{c:splendid1} is extended to local functionals on more general compact manifolds (other than $S^d$).

\subsection{Main results, II: general cancellation phenomenon for local functionals of Uniform Random Waves on compact manifolds}\label{subsec:mainres2}
In the previous subsection, we have demonstrated how \cref{thm:Main1} and its corollaries imply that, when $M=S^d$, the uniform random spherical harmonics $G=\tilde{f}$ induces a {1st and 2nd-order} maximal cancellation on a very large family of geometric functionals (see \cref{c:splendid1}), regardless of them being local (defined by some integral, e.g. volumes and Lipschitz-Killing curvatures) or non-local (e.g. Betti numbers and connected components). We observe once again that this result relies crucially on the irreducibility of the representation induced by spherical harmonics (see \cref{ex:contrexirred}). {Accordingly, any attempt to apply the results of \cref{subsec:mainI} beyond the spherical setting must be restricted to manifolds whose Laplace eigenspaces are irreducible representations. This requirement is highly restrictive and is satisfied only by a limited class of manifolds, including certain homogeneous spaces; see \cite[Theorems 3 and 4]{2019_IrreducibleEigespaces_PetreccaRoser}.} Among the examples considered in this paper, only the spheres $S^d$ fall into this category.

The next statement shows that, if we restrict the class of functionals under consideration to those that are local and of the form \eqref{eq:LocalFunctionalsOfTheValue} below, then a form of maximal cancellation occurs on an arbitrary Riemannian manifold, {the only caveat being} the limitations discussed in \cref{r:afterRW}, Point \ref{itm:HomogeneouSpaz}.
\begin{mainthm}\label{thm:Main2}
    {\it Let the setting of \cref{sec:ourmodel} prevail and let $F\colon \R\to \R$ be a measurable function such that $F(\sqrt{\frac{N}{\vol{} (M)}}\frac{\gamma^1}{\|\gamma\|})$ is square-integrable. Then,
    \be\label{eq:LocalFunctionalsOfTheValue} 
\m E(\tilde{f}):=\int_M F\tyu \tilde{f}(x)\uyt dx
    \ee
    is square-integrable and $\m E(\tilde{f})[2]=0$.}
\end{mainthm}

{The proof of the previous theorem relies on the interplay {between the two conditions \eqref{eq:Yorthogonality} and \eqref{eq:YlmNorm}}, introduced in \cref{sec:ourmodel}.
As already pointed out (see \cref{r:afterRW}, Point \ref{itm:HomogeneouSpaz}), these assumptions restrict the class of Riemannian manifolds to which \cref{thm:Main2} applies. However, as noted at the beginning of the present section, this restriction is substantially weaker than the one required for the application of \cref{c:ofmain1}. {Indeed,} even when $M=G/K$ is a Riemannian homogeneous space (and even if it is symmetric), the eigenspaces of the Laplace operator may fail to be irreducible representations (see again \cite[Theorem 4]{2019_IrreducibleEigespaces_PetreccaRoser}). In such cases, \cref{c:ofmain1} does not apply, whereas \cref{thm:Main2} remains applicable. A concrete example is provided by the torus $M=\mathbb T^d$.
The functional $\m E_u(\lebf)=\m L_d(\lebf \ge u)$, defined as the volume of the excursion set of $\tilde f$, falls within the scope of \cref{thm:Main2}; see \cref{eq:ExcursionArea}. We will discuss this case in detail in \cref{sec:main3}.}

\medskip

The technique employed in the proof of \cref{thm:Main2} allows one to establish a 2nd-order chaotic cancellation also for more complicated classical local functionals, having the following form:
\be\label{eq:funfunfunfunfun} 
\m E(\lebf)=\int_M F\tyu \lebf(x),\|\nabla \lebf (x)\|\uyt dx,
\ee 
with $F\colon \R^2\to \R$ measurable and bounded (or satisfying an adequate integrability condition). 
Prominent examples include the level volume $\m L_{d-1} ( \lebf \ge u)=\frac12\vol{d-1} ( \lebf = u)=\int_M\delta_0(\lebf(x))\|\nabla \lebf(x)\| dx$ and the number of critical points $\m L_0(\nabla\lebf=0)=\int_M\delta_0(\|\nabla \lebf(x)\|) dx$.\footnote{In the two mentioned cases, the $F$ appearing in \cref{eq:funfunfunfunfun} is a distribution and not a function. Nevertheless, one can typically reduce to the case of functions, by standard approximation techniques, or via alternative representation such as \cite{PolyAngst}.} 
For such functionals, we prove \cref{thm:louis2}, reported in \cref{sec:main2} below, under the following additional assumption:
\be \label{eq:homothetic}
\textbf{Homotheticity Assumption: }\quad \|\de_uY(x)\|=\mathrm{const}\cdot \|u\|, \ \forall x\in M, \forall u\in T_xM.
\ee
In the context of \cref{sec:ourmodel}, such requirement is equivalent (see \cref{lem:homothe}) to say that the underlying Gaussian field $f$ should be \emph{homothetic} {(see \cite[Definition 1.11]{cgv2025StecconiTodino}, as well as \cite{elk2024PistolatoStecconi} for the first appearance of such a terminology), a concept meant to generalize the notion of \emph{isotropic field} to general manifolds.}
\begin{remark} One remarkable byproduct of \cref{thm:louis2} is that --- if $\tilde{f}$ is a uniform random wave in the sense of Definition \ref{d:whatisarandomwave} and \eqref{eq:homothetic} is satisfied --- then the second Wiener chaos projection of $\m E(\tilde f)$ in \eqref{eq:funfunfunfunfun} can always be expressed in terms of an appropriate multiple of the harmonic polynomial
\begin{equation}\label{e:thebestharmonicpolynomial}
(x_1,...,x_N)\mapsto \sum_{\ell = 1}^N(x_\ell^2-1)(\lambda^2-\lambda_{i_\ell}^2) =\sum_{\ell = 1}^Nx_\ell^2(\lambda^2-\lambda_{i_\ell}^2) ,
\end{equation}
where ${i_1},...,{i_N}$ is an enumeration of the elements of $\{i\in \N:\lambda_i\in I\}$, and $\lambda^2 = \frac{1}{N} \sum_{\lambda_i\in I} \lambda_i^2$. The polynomial \eqref{e:thebestharmonicpolynomial} plays a crucial role in reference \cite{cgv2025StecconiTodino}.
\end{remark}

\subsection{Main Results, III: a general form of the chaos expansion}\label{subsec:mainres3}

{The two preceding main results (\cref{thm:Main1} and \cref{thm:Main2}) show that, when working with uniform random waves rather than Gaussian ones, the projections of geometric functionals onto the second Wiener chaos vanish under substantially more general conditions. A natural next step is therefore to investigate the variance of these functionals and, potentially, to establish (quantitative) central limit theorems, exploiting e.g. {\it fourth moment theorems} for sequences of random variables living in a fixed chaos; see \cite[Chapter 5]{nourdinpeccatibook}. Although we do not pursue these questions in full generality in the present work, we introduce a general method for computing the exact Wiener chaos expansion of uniform random waves and study its relationship with the corresponding Gaussian framework.}
For conciseness and simplicity, here we will focus on the simplest geometric functional, i.e., the excursion area
\be \label{eq:ExcursionArea}
\m E_u( \lebf ):=\int_M\I{\kop \lebf(x) \ge
 u \pok} dx
;
\ee
an inspection of the proof reveals that the same argument can be extended to other local functionals (i.e., boundary volumes, Euler-Poincaré characteristic), to the price of heavier computations.

The next definition introduces a crucial technical object.
\begin{definition}\label{def:gammax}
Let the setting of \cref{sec:ourmodel} prevail. For every $x\in M$, we define the $N$-dimensional Gaussian vector $\gamma_x$ as
\be \label{eq:gammaperp}
\gamma_x:=\gamma-f(x) Y(x)\sqrt{\frac{\vol{}(M)}{N}}.
\ee
\end{definition}

\smallskip 
For any fixed $x\in M$, the Gaussian vector $\gamma_x$ 
is the orthogonal projection in $\R^N$ of $\gamma$ onto the hyperplane $Y(x)^\perp$, therefore it is a standard Gaussian vector of rank $N-1$ supported on such space. 
Because of this, $\gamma_x$ and $f(x)=cY(x)^T\gamma$ are independent. 
Notice that since
\be 
\|f\|_{L^2(M)}=\|\langle \gamma,Y(\cdot)\rangle\|_{L^2(M)}\sqrt{\frac{\vol{} (M)}{N}}=\|\gamma\| \sqrt{\frac{\vol{} (M)}{N}},
\ee 
we can write the uniform random field as follows:
\be \label{eq:therootofalltricks}
\tilde{f}(x)=\frac{c f(x)}{\sqrt{f(x)^2+\|\gamma_x\|^2}},\quad \text{where}\quad c:=\sqrt{\frac{N}{\vol{} (M)}}.
\ee
The advantage of such a {representation} is that, by construction, $f(x)$ and $\gamma_x$ are independent standard Gaussian random vectors, respectively, in $\R$ and in $Y(x)^\perp$.} This fact will become particularly relevant by virtue of \cref{sec:magictrick}.
 \begin{remark} 
Note that, given any point $x\in M$, the pair of random vectors $(f(x),\gamma_x)$ completely determines both $f(\cdot)$ and  $\lebf(\cdot)$ as random fields over $M$. 
In particular, $\gamma_x$ can be identified with the conditional random field $[f(\cdot)|f(x)=0]$.
\end{remark}

The process $x\mapsto \gamma_x$ introduced in \eqref{eq:gammaperp} appears naturally in chaos computations, and most of the following formulas are given in terms of it, starting from forthcoming Proposition \ref{thm:Main3}, yielding a clear way to compare the chaos expansion of the excursion indicator functions $\I\{ \lebf(x) \ge
 u \}$ (uniform case) and $\I\{ f(x) \ge
 u \}$ (Gaussian case).
To this end, let us recall that for the (unit-variance) Gaussian field $f$, we have the following Wiener chaos expansion:
\be \label{eq:macaron}
\I\{ f(x) \ge
 u \}=\sum_{q\in \N}J_q\tyu u\uyt \frac{H_{q}\tyu f(x)\uyt}{q!}
\ee
where the sequence $J_0(u)=\Phi(u)$, and $J_q(u)=(-1)^{q-1}H_{q-1}(u)\frac{e^{-\frac{u^2}{2}}}{\sqrt{2\pi}}$, for $q\ge 1$, see \cite{MaWi2011excarea}.
\begin{proposition} \label{thm:Main3} 
Let the setting of \cref{sec:ourmodel} prevail. For any {$u\in (-\sqrt{\NN},\sqrt{\NN})$}, the following expansion converges in $L^2$:
\bega \label{eq:almostDecomp}
\I{\kop \lebf(x) \ge
 u \pok}
 &=
 \sum_{q\in \N} 
 {J_q \tyu \|\gamma_x\|  \frac{u}{\sqrt{\frac{N}{\vol{}(N)} - u ^2}}  \uyt}
 \frac{H_{q}\tyu f(x)\uyt}{q!},
\eega
where the sequence $J_q(.)$ has been defined above. {Moreover, the term 
 ${J_q \tyu \|\gamma_x\|  \frac{u}{\sqrt{\frac{N}{\vol{}(M)} - u ^2}}  \uyt}$
 converges to $J_q(u\sqrt{\vol{}(M)})$ in $L^2$ as $N\to +\infty$.}
\end{proposition}

\begin{remark}
We stress that \cref{eq:almostDecomp} is not (yet) a Wiener chaos decomposition, see \cref{sec:Wichaos}: for any $q\in \N$, the term $J_q
 \tyu \|\gamma_x\| \sqrt\frac{v^2}{{N-v^2}} \uyt$ {(written in terms of the  \emph{effective level} $v = v(u) := u\sqrt{\vol{}(M)}$)} does not belong to a fixed chaos of $\gamma$. We will indeed derive the full Wiener chaos decomposition of $\I{\kop \lebf(x) \ge
 u \pok}$ by expanding such term into its chaos components --- see \cref{t:wienerchaos} below. 
\end{remark}

\begin{remark}
It would be tempting to derive the Wiener chaos expansion for $\I{\kop \lebf(x) \ge u \pok}$ by viewing it as the indicator of a Gaussian field evaluated at a random threshold, namely by exploiting the identity
\[
\I{\kop \lebf(x) \ge u \pok}
=
\I{\kop f(x) \ge
\frac{\|\gamma\|}{\sqrt{N}} \sqrt{\vol{}(M)}\,u \pok},
\]
where, as before, $f(x)$ denotes a Gaussian eigenfunction with unit variance. Here, the factor $\sqrt{\vol{}(M)}$ stems from our normalization conventions (recall that $\tilde f$ has unit $L^2$ norm, whereas $f$ has unit pointwise variance; see \cref{rmk:normalization}), while the prefactor $\|\gamma\|/\sqrt{N}$ converges to one by the law of large numbers as $N\to\infty$. One might then naively expect that the following $L^2$ expansion holds:
\bega
\I{\kop \lebf(x) \ge u \pok}
&=
\sum_{q\in\N} J_q
\tyu \frac{\|\gamma\|}{\sqrt{N}} \sqrt{\vol{}(M)}\,u \uyt
\frac{H_q\tyu f(x)\uyt}{q!}.
\eega
However, Proposition \ref{thm:Main3} shows that this heuristic reasoning is flawed: the eigenfunctions appearing on the left-hand side of the inequality are not independent of the (random) threshold on the right-hand side. As a consequence, the latter cannot be substituted directly into the standard expressions for the Wiener chaos decomposition of the excursion area in the Gaussian setting. The proofs below develop a general technique to overcome this difficulty and to derive valid Wiener chaos expansions within the present framework.
\end{remark}

We refer the reader to \cref{subsec:notaintegrals} for the notation adopted in the next statement.

\begin{mainthm}[Wiener Chaos decomposition]\label{t:wienerchaos} {\it Let the setting of \cref{sec:ourmodel} prevail, and let $\gamma_x$ be as in \cref{def:gammax}. For any $u\in (-\sqrt{\NN},\sqrt{\NN})$, the following expansion converges in $L^2$:
    \begin{align} 
\mathbb I \left\{ \lebf(x) \ge u \right\} & = \sum_{q\in\N}\mathbb I \left\{ \lebf(x) \ge u \right\}[q] \\
& = \sum_{q\in\N} \sum_{\substack{i=0,\\ \text{even }}}^q \mathfrak C_N(q,i,u) \times \int_{S(Y(x)^\perp)} \frac{H_{i}(\gamma_x^Tw)}{\sqrt{i!s_{N-2}s_{N-3}B(\frac{i+1}{2},\frac{N-2}{2})}}
 \frac{H_{q-i}(f(x))}{\sqrt{(q-i)!}} dw,
\end{align}
where $\{\mathfrak C_N(q,i,u)\}_{q,i\in\mathbb N, i\le q} \in \ell^2(\mathbb N)$ is a square-summable sequence of real numbers, defined as \begin{equation}
    \mathfrak C_N(q,i,u) = \sqrt{\frac{s_{N-2}}{i!s_{N-3}B(\frac{i+1}{2},\frac{N-2}{2})}}\E \left[ J_q
 \tyu \|\eta\| \sqrt{\frac{v^2}{{N-v^2}}} \uyt H_i(\eta_1)\right],
\end{equation} 
with $\eta\sim\mathcal N(0,1_{N-1})$, $v=v(u)=\sqrt{\vol{}(M)}u$, and we have introduced the usual Beta function, defined by
\[
B\tyu \frac{i+1}{2},\frac{N-2}{2}\uyt
=\frac{\Gamma(\frac{N-2}{2})\Gamma(\frac{i+1}{2})}{\Gamma(\frac{N+i-1}{2})}.
\]
}
\end{mainthm}

\begin{remark}Because of the presence of a spherical integral, the normalization of the Hermite polynomials in the series expansion of Theorem \ref{t:wienerchaos} is not the square root of a factorial as usual, but it is chosen to ensure that $\|\mathfrak C_N(\cdot,\cdot,u)\|_{\ell^2(\N)}^2$ {is the global variance of $\mathbb I \left\{ \lebf(x) \ge u \right\}$}; see \cref{lem:SphericalChaos}
below.
\end{remark}

The derivation of the Wiener chaos expansion for the excursion area of uniform random waves makes the cancellation of the second-order chaos explicit; see \cref{prop:secondchaos}. Remarkably, this cancellation arises from an exact compensation between the fluctuations of its two constituent terms: one depending on $f(x)$ and the other on $\|\gamma_x\|$, as defined in \eqref{eq:gammaperp}. Beyond this structural insight, the availability of an explicit chaos decomposition enables more refined results on variance asymptotics. In particular, since the cancellation occurs for every threshold level $u$, one might conjecture that the asymptotic variance exhibits the same rate for all $u$. {However, this is not the case on the two-sphere $M=S^2$.} Indeed, in this framework a form of second-order Berry cancellation persists {at the level of the fourth chaotic projection}: in the nodal case ($u=0$), the variance is of strictly smaller asymptotic order than for (almost) any other threshold $u\neq 0$. In addition, two further cancellation levels arise.

\begin{proposition} \label{p:lowerbound} Let the setting of \cref{sec:ourmodel} prevail; suppose that $M=S^2$, { that $\lebf$ is a uniform random wave of eigenvalue $-\ell(\ell+1)$ and let $N=2\ell +1$}. Let $\m E_u(\tilde f )$ be defined as in \cref{eq:ExcursionArea}.  
Then, 
as $N\to\infty$, the fourth-order chaos of the excursion area has variance
\begin{align} \notag
     \Var\left(\m E_u(\tilde f )[4] \right) &=   \frac{J_4\left(u\sqrt{\vol{}(M)}\right)^2}{(4!)^2} \Var \left( \int_M H_4(f(x)) dx \right) + o \left( \Var \left( \int_M H_4(f(x)) dx \right)\right) \\
    & = {e^{-{u^2}\vol{}(M)}H_3\tyu u\sqrt{\vol{}(M)}\uyt^2   D\frac{\log N}{N^2} + O \left( \frac{1}{N^2} \right) ,\label{e:rrr}}
\end{align}
for some positive constant $D>0$. Moreover, $\Var\left(\m E_u(\tilde f )[q] \right)=O(N^{-2})$ for all $q\geq 3$, $q\neq 4$.
\end{proposition}

\begin{remark}
    Note that the expression $H_3( u\sqrt{\vol{}(M)})$ is a degree three polynomial in $u$, vanishing at $u=0$ and in two other points, which we will denote as $u_1,u_2$. It is possible to trace the exact value of the constant $D$ from the proof, see \cref{sec:main3}.
\end{remark}

\begin{remark}
    Because of the orthogonality of chaotic projections, the right-hand side of \eqref{e:rrr} provides a lower bound for the variance of the excursion area $\m E_u(\tilde f )$, whenever $u\notin \{0, u_1,u_2\}$.
\end{remark}

In the next Remark, we recall for convenience the variance asymptotics for the different chaoses in the Gaussian case. 

\begin{remark}[Variance asymptotics in the Gaussian case]\label{VarAsyGau} In the Gaussian case, we have that 
\begin{equation}
    \Var(\m E_u(f)[q]) = \frac{J_q(u)^2}{(q!)^2} \Var\left(\int_{S^2} H_q(f(x))dx  \right),
\end{equation}
see also \cref{eq:macaron}. Notice that $J_q(u)=0$ if and only if $H_{q-1}(u)=0$; hence, $J_q(0)=0$ if $q$ is even. In \cite{MaWi2011excarea}, it is proved that \begin{align}
        & \Var\left(\int_{S^2} H_2(f(x))dx  \right) = c_2 \frac{1}{N}, \\
        & \Var\left(\int_{S^2} H_4(f(x))dx  \right) = c_4 \frac{\log N}{N^2} ,\\
        & \Var\left(\int_{S^2} H_q(f(x))dx  \right) = c_q\frac1{N^2}, \qquad q\ge 3, q\neq 4, \\
        & \Var\left(\sum_{q\ge 3, q\neq 4}  \frac{J_q(u)}{q!}\int_{S^2} H_q(f(x))dx  \right) = C \frac1{N^2},\label{eq:oddchaoses}
    \end{align} 
    for some positive constants $c_q, \ q \in \mathbb{N}$ and $C$.
    \end{remark}
    \begin{remark}
    Note that the even-order chaotic projections $\m E_u(\tilde f )[2q]$ (uniform case) and $\m E_u( f )[2q]$ (Gaussian case) for $u=0$ are both identically zero for all $q$. This is also explained by the symmetry of the functional $\m E_0$, which is odd. Since $\m E_0(f) = \m E_0(\tilde f)$, the variance asymptotics extend to the uniform case: $\Var(\m E_0 (\tilde f)) = \Var(\m E_0 ( f))= O(N^{-2})$, where the last equality follows from \cref{eq:oddchaoses}. For non-zero levels $u$, in the Gaussian case we have $\Var (\m E_u(f)) = O(N^{-1})$: this is due to the presence of the second chaos that \emph{masks} the contribution of the fourth chaos, which has variance $O(N^{-2} \log N )$, and of the sum of all the other chaoses, which have variance $O(N^{-2})$.
\end{remark}

\section{Main Result I: Formal Statements and Proofs}\label{sec:main1}

In this section we prove \cref{thm:Main1} and its two consequences, \cref{c:splendid0} and \cref{c:splendid1}. \cref{thm:Main1} is a general criterion to infer the cancellation of the second chaos of general functionals of a uniform random field, whenever it is compatible with an irreducible group representation. We obtain \cref{c:splendid0} and \cref{c:splendid1} by applying \cref{thm:Main1} to the case of uniform random spherical harmonics. 
The same ideas could be extended to monochromatic uniform random waves on more general homogeneous spaces; however, this case is addressed in forthcoming Section \ref{sec:main2} with somewhat different techniques. 

\cref{thm:Main1} applies naturally to the investigation of random spherical harmonics, our prototype case. Thus, it is natural to start this section by recalling some basic facts about such a setting. 

\subsection{Representations and spherical harmonics}
Let $\m H_{\ell,d+1}\subset L^2(S^{d})$ be the vector space generated by the spherical harmonics $Y_{\ell,\bullet}$ of degree $\ell$, see \cref{r:afterRW}.(b) and \cref{def:sph_harm_HqN}; in other words, $\m H_{\ell,d+1}=\ker(\Delta+\lambda_\ell^2)$, where $\lambda_\ell^2=\ell(\ell+d-1)$. In this paper we always work under the identification $\R^{N_\ell}\cong \m H_{\ell,d+1}$, with $N_\ell$ as in \cref{e:sphericalharmonics}; consequently $SO(N_\ell)$ is identified with the space of linear orientation-preserving $L^2$-isometries of $\m H_{\ell,d+1}$. Taking any $R\in O(d+1)$, we obtain by pull-back an isometry $\f\mapsto R^*\f$ of $\m H_{\ell,d+1}$ as follows:
\be 
(R^*\f)(x)=\f(Rx),\quad \forall x\in M,\ \f\in V_\ell
\ee
The isometry $R^*$ is thus an element of $SO(N_\ell)$ and the function
\be \label{eq:reprSOdN}
L\colon SO(d+1)\to SO(N_\ell),\quad R\mapsto L(R):=R^*
\ee
is a group homomorphism, i.e., a representation of $SO(d+1)$. For any $v\in \R^{N_\ell}$, the set
\be 
O(v)=\kop R^*v : R\in SO(d+1)\pok
\ee
is called the \emph{orbit} of $v$; it is a subset of $\R^{N_\ell}$. We recall that a group representation $L\colon G\to SO(N)$ is said to be \emph{irreducible} if and only if there are no non-trivial vector subspaces that are fixed by $L$; equivalently, if every non-zero orbit spans the whole space, see \cite{FarautBook}.
\begin{proposition}\label{prop:Lirrep}
    The representation $L$ defined at \cref{eq:reprSOdN} is irreducible: every orbit $O(v)$ spans the whole space $\R^{N_\ell}$:
    \be 
\mathrm{span}\ O(v)=\R^{N_\ell}, \quad \forall v \in \R^{N_\ell}\backslash\{0\}.
    \ee
\end{proposition}
\begin{proof}
The result follows directly from 
 \cite[Theorem 9.3.4]{FarautBook}. (See also \cite[Prop. 3.27]{Marinucci_Peccati_2011} for a more explicit proof in the case $d=2$.)
\end{proof}
{
\subsection{Preliminary remarks on invariance properties}
{Let us recall our main setting in the spherical case $M=S^d$, discussed in \cref{sec:ourmodel}. We consider the uniform random field $\tilde f$ defined by
\be 
\tilde{f}_\ell(\cdot)=\left\langle \frac{\gamma}{\|\gamma\|}, Y(\cdot) \right\rangle,
\ee
where $Y(\cdot)=Y_{\ell,\bullet}$ denotes the vector of spherical harmonics introduced in \cref{e:sphericalharmonics}, and $\gamma$ is a standard Gaussian vector in $\R^{N_\ell}$.
We focus first on the excursion area functional, which we view as a functional of $\gamma$:
\be 
E_u(\gamma)=\int_M \I\kop \left\langle \frac{\gamma}{\|\gamma\|},Y(x)\right\rangle \ge u \pok \, dx.
\ee
As a preliminary observation, note that $E_u(g\gamma)\neq E_u(\gamma)$ for a generic rotation $g\in SO(N_\ell)$, even though the law of the two functionals is the same. Nevertheless, the excursion area enjoys a more structured form of rotational symmetry. Recalling the definition of the representation $L$ in \cref{eq:reprSOdN}, we have $Y(Rx)=R^*Y(x)$, which immediately implies that the excursion area functional is $L$-invariant, namely
\be 
E_u\tyu R^* \gamma\uyt = E_u\tyu \gamma\uyt, \quad \forall R\in SO(d+1).
\ee
Moreover, since $E_u$ depends only on the uniform random wave $\tilde f$, it is positively $0$-homogeneous (cf. \cref{d:0hom}):
\be
E_u(t\gamma)=E_u(\gamma), \quad \forall t>0.
\ee
}

These two properties---$L$-invariance and positive $0$-homogeneity---are precisely the assumptions required in \cref{thm:Main1} and Corollary \ref{c:ofmain1}, showing that, when the representation $L$ is irreducible, these properties alone imply the cancellation of the second Wiener chaos projection of $E_u(\gamma)$. Although we have illustrated $L$-invariance and positive $0$-homogeneity here through the specific example of the excursion area, the same structural features will be shown to hold for a much broader class of functionals. In particular, \cref{c:splendid0} and \cref{c:splendid1} follow by verifying that the same two properties are satisfied by all functionals appearing therein, combined with the irreducibility of the representation $L$ established in \cref{prop:Lirrep}.

}
{

\subsection{Further notation}\label{ss:furthnot} Let $\gamma\sim\mathcal{N}(0, \mathbb{I}_N)$, and consider a (not necessarily $0$-homogeneous) square-integrable random variable $E(\gamma)$. As in the statement of Theorem~\ref{t:nonmale}, and using 
Remark~\ref{r:doh}-(b), we set
\begin{equation}\label{e:tensorq}
  K_q := K_{E(\gamma)[q]}, \qquad q \ge 1,
\end{equation}
that is, $K_q \in \mathbb{T}(q,N)$ is the unique symmetric tensor for which 
\eqref{e:zorro} holds.
 For $q\geq 1$, we also define the homogeneous polynomial
\begin{equation}\label{e:hompq}
    P_q(x): =  q!\sum_{(i_1,..,i_q)\in[N]^q } K_q(i_1,...,i_q)\, x_{i_1}\cdots x_{i_q}, \quad x = (x_1,...,x_N)\in \R^N,
\end{equation}
and we observe that \eqref{e:innertenso} yields the identity
\begin{equation}\label{e:hompqS}
    P_q(v)= \mathbb{E}\big[E(\gamma)[q] H_q(\langle \gamma,v\rangle)\big] = \mathbb{E}\big[E(\gamma) H_q(\langle \gamma,v\rangle)\big] , \quad v \in S^{N-1},
\end{equation}
where we have used the fact that, for all $v= (v_1,...,v_N)\in S^{N-1}$, 
$$
C_q\ni H_q(\langle \gamma,v\rangle) = \sum_{(i_1,..,i_q)\in[N]^q } (v_{i_1}\cdots v_{i_q})\,\,: \gamma^{i_1}\cdots \gamma^{i_q} : \,\,.
$$  
It is clear that $E(\gamma)[q] = 0$ if and only of $P_q(v)=0$, for all $v\in S^{N-1}$.

}

{
\subsection{Proof of Theorem \ref{thm:Main1}}
The fact that ${\rm Tr}\, K_2 = 0$ is a direct consequence of $0$-homogeneity and Theorem \ref{t:nonmale}. To deal with the second part of the statement, we adopt the notation introduced in Section \ref{ss:furthnot}. Using the $L$-invariance of $E$, we infer that, for all $q\geq 1$ and all $R\in G$,
\bega \label{e:jott1}
P_q(R^*v)
&=
\E\kop E(\gamma) H_q\tyu \langle \gamma ,R^*v\rangle \uyt \pok
=
\E\kop E(\gamma) H_q\tyu \langle (R^*)^T\gamma ,v\rangle \uyt \pok 
\\
&=
\E\kop E(R^*\gamma) H_q\tyu \langle \gamma ,v\rangle \uyt \pok = 
P_q(v).
\eega
The last relation implies that the orbit $O(v)$ of any $v\in \R^N$ is contained in a level set of $P_q$, from which we deduce that
\be\label{eq:orbit1} 
\forall v\in S^{N-1} \,\, \mbox{such that}\,\, P_q(v) = c, \,\, \mbox{one has that}\,\, O(v) \subset P_q^{-1}(c)\cap S^{N-1}.
\ee
We also notice that relations \eqref{e:jott1}--\eqref{eq:orbit1} yield that the mapping $ v\mapsto P_1(v)$ is the restriction to the sphere of a linear application having the form
$$
\R^N \to \R : y\mapsto a\cdot y := \langle a,\, y\rangle,
$$
for some $a\in \R^N$ such that $R^* a = a$ for every $R \in G$; since $\{t\,a : t\in \R \}$ is $L$-invariant, if $L$ is irreducible one must necessarily have $a=0$, and therefore $P_1(v) = 0$ for all $v\in S^{N-1}$. To conclude, we observe that, since \eqref{e:jott1} is in order, one has that, for all $R\in G$,
$$
(R^*v)^T K_2 (R^*v) = v^T K_2v, \,\,\forall v\in S^{N-1} \quad \Longleftrightarrow\quad K_2 R^* =  R^*K_2,
$$
that is, $K_2$ commutes with $L$. Since $K_2$ is symmetric and the $L_\ell$s are pairwise non-equivalent, one can use Schur's Lemma (see \cite[Theorem 6.1.3-(i)]{FarautBook}) to deduce that, for coefficients $\alpha_1,...,\alpha_M \in \R$,
$$
K = \bigoplus_{\ell = 1}^M \, \alpha_\ell\, \mathbb{I}_{N_\ell},
$$
and the conclusion follows at once.}

\subsection{Isotropic and semialgebraic singularities}
{In this subsection, we explain and discuss the assumptions of \cref{c:splendid0}.}
\subsubsection{Jets}
{Recall that a Riemannian manifold $(M,g)$ has an associated Levi-Civita connection $\nabla$, which allows one to define the \emph{covariant derivative} of any tensor, analogously to the standard calculus on $\R^d$; this is a standard framework, for which we refer to the monograph \cite{leeriemann}. Given a smooth function $f\in \mC^\infty(M)$ and an integer $k\in \N$, we define the \emph{$k$-jet of $f$ at $x$} (see \cite[Section 4]{Hirsch}) of $f$ as the tuple:
\be \label{eq:jet}
j^k f (x) := \tyu x, f(x),\nabla_xf,\dots,\nabla_x^{(k)}f \uyt, 
\ee
where $\nabla^{(k)}_xf$ is the $k$-multilinear form on $T_xM$ (which we identify with $T^*_xM$ via the metric) obtained after $k$ covariant differentiations of $f$.}\footnote{{Normally, jets are defined independently of the metric, as equivalence classes of maps, see \cite[Section 4]{Hirsch}. On a Riemannian manifold, i.e., once a metric is fixed, one can use the Levi-Civita connection to have a more direct representation such as that of \cref{eq:jet}.
}}
{The natural intuitive interpretation of jets is that $j^kf(x)$ is a coordinate-free version of the Taylor polynomial of $f$ at $x$ and indeed the space of all $k$-jets at $x\in M$, denoted $J^k_x(M)$ (see \cite[Section 4]{Hirsch}) is isomorphic to the space of degree $k$ polynomial functions on $T_xM$. So, $j^kf(x)\in J^k_x(M)$ and we call the function $x\mapsto j^kf(x)$ the \emph{$k$-jet prolongation of $f$}; this is a section of the vector bundle over $M$ having $J^k_x(M)$ as fiber over $x$, denoted as $J^k(M)$ and called the \emph{space of $k$-jets (from $M$ to $\R$)}. Such notation extends, with obvious generalization, to $j^kf(x)\in J^k(M,\R^a)$ for a vector valued function $f\in \mC^\infty(M,\R^a)$. We refer to \cite[Section 4]{Hirsch} for a detailed treatment.}

{
For instance, the jet prolongation of the Gaussian field $f=\langle \gamma, \y\rangle$, that we consider in \cref{sec:ourmodel}, is given by
\bega\label{eq:jetnablay} 
j^k f (x) = \tyu x, \langle \gamma, \y(x) \rangle, \langle \gamma, \nabla_x\y \rangle,\dots, \langle \gamma, \nabla_x^{(k)}\y \rangle \uyt,
\eega
and defines a Gaussian random section of $J^k(M)$.}
{\begin{remark}
We will never really need to manipulate directly formulas such as that in \cref{eq:jetnablay} in this paper, because our argument for the cancellation relies on the abstract general statement of \cref{thm:Main1}. Instead, we need the language of jets only to identify precisely the class of functionals falling into the hypotheses of \cref{c:splendid0}.
\end{remark}}
{\subsubsection{On the importance of jets}
The language of jets is very useful to describe a broad class of functionals of a geometric nature.  In fact, all of those that are typically studied in the context of Gaussian random fields, including all those considered in this paper, fall under the class of type-$W$ singularities, introduced in \cite{2022BreKeneLer, witdoash2021LerarioStecconi} and defined in terms of jets. In its most general version, a \emph{type-$W$ singularity of $f$} is just the preimage of a subset $W\subset J^k(M)$ under the $k$-jet prolongation of $f$; we will use the notation:
\be \label{eq:typeW}
Z_W(f):=(j^kf)^{-1}(W).
\ee 
Below, we show that excursion sets, level sets, and critical points admit such a description. This fact is straightforward and has been already exploited in many contexts, including that of Gaussian random fields, see \cite{KR2022Stecconi,mttps2024LerarioStecconi, gts2025LerarioMarinucciRossiStecconi, lerario2021probabilistic}. Having such a unifying abstract description has a significant methodological advantage: it is often the case that the same proof applies to a wide variety of functionals, at the same time. This is exactly the case of \cref{c:splendid1}. 
Here is a list of prominent examples. The first two show how \cref{c:splendid0} implies \cref{c:splendid1}, see also \cite[5.1 Examples]{gts2025LerarioMarinucciRossiStecconi}.}
{
\begin{example}\label{ex:typeW}
\begin{enumerate}[(i.)]
    \item \emph{Excursion set and level sets}. Taking $k=0$ we have that $J^0(M)=M\times \R$, so that
    \be 
    \{x\in M: f(x)\ge u\}= Z_{M\times [u,+\infty)}(f), 
    \quad \text{and}\quad  
    \{x\in M: f(x)= u\}= Z_{M\times \{u\}}(f). 
    \ee
    \item \emph{Critical points}. Taking $k=1$, we have that $J^1(M)=TM\times \R$, so that 
    \be 
\text{Crit}(f):=\kop x \in M : \nabla_x f=0\pok=Z_{0_{TM}\times \R}(f),
    \ee
    where $0_{TM}=\{(x,v)\in TM: x\in M, v=0\}\cong M$ denotes the zero section of $TM$.
    \item \emph{Tangencies}. Let $N\subset M $ be a smooth submanifold of $M$. Then 
    \be 
\{x\in M: Z_{\{u\}}(f) \text{ is tangent to } N \text{ at $x$}\}= \{x\in M\colon \nabla_x f \subset T_xN\}
    \footnote{Here we are using the word \emph{tangent} as a synonymous of \emph{non-transverse}.}
    \ee
    is clearly a set of the form $Z_{W(N)}(f)$ for some $W(N)\subset J^1(M)$, determined by $N$. 
    \item \emph{Inflection points}. There are many possible example of singularities extending the notion of the \emph{flex of a curve}, or inflection point. These are characterized by some degeneration condition on the curvature of, for instance, $Z_{\{u\}}(f)$, hence, by some $W\subset J^2(M)$. We give the example of \emph{minimal points}, that is, points of $Z=Z_{\{u\}}(f)$ where the trace of the second fundamental form of $Z$ is zero, i.e., where the mean curvature vanishes. Then,
    \be 
\{x\in M: x \text{ is a minimal point of } Z_{\{u\}}(f)\}=Z_{W_{\text{mean}=0}}(f),
    \ee
    for some $W_{\text{mean}=0}\subset J^2(M)$. Indeed, the formula computing the mean curvature $H_x$ of $Z_{\{u\}}(f)$ at $x$ (see \cite{leeriemann}):
    \be 
H_x=\frac{
\mathrm{Hess}_x f \tyu \nabla_x f, \nabla_x f\uyt
}
{\|\nabla_x f\|^{3}}
-\frac{
\Delta f(x)
}{
\|\nabla_x f\|
},
\ee
is a semialgebraic expression involving derivatives of $f$ up to the second order. Note that such a formula is built in the definition of the set $W_{\text{mean}=0}$.
    \item \emph{Morse bootstrap trick.}
    Let $h\in \mC^\infty(M)$, then the set
    \be 
\text{Crit}\tyu 
h|_{Z_{\{u\}(f)}}
\uyt
=\kop x\in M: f(x)=u,\ \nabla_xh \wedge \nabla_xf=0\pok
    \ee
    can either be interpreted as a $Z_{W(h)}(f)$, for some $W(h)\subset J^1(M)$, but also as a $Z_{\{u\}'}(h,f)$, for some $\{u\}'\subset J^1(M,\R^2)$ (the space of $k$-jets from $M$ to $\R^2$). In the second case, $W=\{u\}'$ does not depend on $f$ nor $h$. This kind of singularity is very useful in combination with Morse theory, which entails that, if $Z=Z_{\{u\}}(f)$ is a submanifold and $h|_Z$ is a Morse function, then by Morse inequality, we have
    \be 
\sum_{i=1}^d b_i\tyu Z_{\{u\}}(f)\uyt\le \# \tyu Z_{\{u\}'}(h,f) \uyt
    \ee
This idea can be pushed substantially further, see \cite[Theorem 2.3]{mttps2024LerarioStecconi}, which extends the Morse inequalities to the following statement (here reported informally): for any $W\subset J^k(M,\R^a)$ there exists $W'\subset J^{k+1}(M,\R^{a+1})$ such that the Betti numbers of $Z_W(f)$--a subtle non-local topological quantity--are bounded by the cardinality of $Z_{W'}(h,f)$, for \emph{almost every} choice of auxiliary function $h$. This method was used in \cite{GaWe3,mttps2024LerarioStecconi,gayet_2024_ALG} for bounding the expectation of Betti numbers of Kostlan polynomials. 
\item \emph{Lipschitz-Killing bootstrap trick}. 
We refer to \cite{AdlerTaylor} for a complete treatment of Lipschitz-Killing curvatures and just recall that given a regular enough subset $Z\subset M$ of dimension $n$, its Lipschitz-Killing curvature are real numbers $\m L_0(Z),\m L_1(Z),\dots, \m L_n(Z)$, with $\m L_n(Z)=\vol{n}(Z)$ defined by integrating certain polynomial functions of the (intrinsic and extrinsic) curvature of $Z$. Using a strategy similar to that adopted for Betti numbers, one can reduce the study of these quantities for a given $W\subset J^k(M)$ from submanifolds to point processes, via a formula of the form:
\be 
\m L_i\tyu Z_W(f)\uyt= \E \sum_{x\in Z_{W_i}(h,f)}\a_i\tyu j^{k+1}_x(h,f)\uyt,
\ee
where $W_i\subset J^{k+1}(M,\R^{1+a_i})$, $\a_i\colon W_j\to \R$, and a random field $h\randin \mC^\infty(M,\R^{a_j})$ are suitably defined. Such a methodology have been rigorously proved and applied in \cite{gts2025LerarioMarinucciRossiStecconi}, in the case $d=2$, moreover this idea is at the core the proof of \cite[Theorem 13.3.1-2]{AdlerTaylor}, which constitutes one of the technical pillars of the whole celebre book \cite{AdlerTaylor}. 
\end{enumerate}
\end{example}
}
{\subsubsection{Isotropic singularities}
Let us now restrict to the case $M=S^d$. 
\begin{definition}[Isotropic singularity]\label{def:isosing}
We say that a subset $\Sigma \subset J^k(S^d,\R^a)$ is \emph{isotropic} if for any $R\in SO(d+1)$, we have that 
\be\label{eq:isotropR} 
j^kf (Rx)\in \Sigma \iff j^k(f\circ R)(x)\in \Sigma,
\ee
for all $x\in S^d$. In this case, we say that $Z_\Sigma(f)=j^kf^{-1}(\Sigma)$ is an \emph{isotropic singularity}.
\end{definition}}
This notion is a weaker analogue of that of \emph{intrinsic singularity}, introduced in \cite[Definition 2.6]{mttps2024LerarioStecconi}. It is immediate to see from such definition, that any intrinsic singularity is also isotropic\footnote{\cite[Definition 2.6]{mttps2024LerarioStecconi} says that a subset $\Sigma \subset J^k(M)$ is an intrinsic if it satisfies a relation analogous to \cref{eq:isotropR} for any diffeomorphism $R$.}, hence, the statement of \cref{c:splendid0} is stronger than if $\Sigma$ were assumed to be intrinsic.
The spirit of such a concept is that the vast majority of type-$W$ singularity (see \cref{eq:typeW}) arising ``naturally'' in differential geometry turns out to be intrinsic and therefore isotropic. 


\subsubsection{Semialgebraic singularities}
A subset $A\subset \R^n $ is said to be \emph{semialgebraic} if it can be expressed as finite unions and intersections of polynomial inequalities.
\begin{definition}\label{def:semialg}
A \emph{semialgebraic} set $A\subset \R^n$ is a set of the form
\be \label{eq:semialgdef}
A=\bigcup_{i=1}^b\bigcap_{j=1}^a \kop x\in \R^a\ \big|\ 
\text{sign} \tyu p_{ij}(x)\uyt =\e_{ij}
\pok,
\ee
where $a,b\in\N$ are finite,  $\e_{ij}\in\kop -1,0,1\pok$ and $p_{ij}$ are real polynomials in $n$ variables.
\end{definition}
{The theory of semialgebraic sets is a very rich and well developed classical topic, with many deep results, for which we refer to \cite{RAG}. In particular, two pillars of such a theory are the Tarsky-Seidenberg theorem (\cite[Theorem 1.4.2]{RAG}), which makes it very easy to recognize when a set if semialgebraic, and \cite[Theorem 9.7.11]{RAG}, establishing that any semialgebraic set admits a Whitney stratification. The role of semialgebraic geometry in our paper is subtle, but in full analogy with past works e.g. \cite{mttps2024LerarioStecconi,lerario2021probabilistic,gts2025LerarioMarinucciRossiStecconi,2022BreKeneLer}. For most ``natural'' type-$W$ singularities, the defining subset $W\subset J^k(M)$ is semialgebraic. For instance, all the examples given in \cref{ex:typeW} are semialgebraic\footnote{With the exception of $W(h)$, which is semialgebraic if $h$ is a semialgebraic function.}. This simple observation (made such by the Tarsky-Seidenberg theorem) allows to exploit the vast plethora of deep results from semialgebraic geometry, that could otherwise pass unnoticed. In the following we will rely on the fact that if $Z\subset S^d$ is a semialgebraic set of degree $d$, then its Lipschitz-Killing curvatures and Betti numbers are well defined (because $Z$ is Whitney stratified) and they are bounded.
}

{
Clearly, the sphere $S^d$ is semialgebraic subset of $\R^{d+1}$ and $J^k(S^d,\R^a)$ is also semialgebraic, in the following sense. The space $J^k(\R^{d+1},\R^a)$ is canonically isomorphic to the vector space 
\be\label{eq:V} 
V=\R^{d+1}\times \R[v_0,\dots,v_d]^a_{\le k},
\ee 
where  $\R[v_0,\dots,v_d]^a_{\le d}$ is the space of $a$-tuples of polynomials of degree at most $k$ in $d+1$ variables. 
The space of jets $J^k(S^d,\R^a)$ can be canonically identified as the subset of $J^k(\R^{d+1},\R^a)$, consisting of all pairs $j^kF(x)=(x,P)$ of where $x\in S^d$ and $P$ is the Taylor polynomial at $x$ of a $0$-homogeneous function $F$ (as in \cref{d:0hom}), i.e., satisfying $P(v)=P(v-xx^Tv)\ \forall v\in \R^{d+1}$. Combining these two identifications, we obtain that $J^k(S^d,\R^a)$ is a semialgebraic subset of the vector space $V$. 
\begin{definition}[Semialgebraic singularity]\label{def:semialgsing}
We say that a subset $\Sigma\subset J^k(S^d,\R^a)$ is semialgebraic, if it is semialgebraic in $V$ (see \cref{eq:V}). In such case, if moreover $f\in \mC^\infty(S^d)$ is the restriction of a polynomial in $\R^{d+1}$, we say that $Z_\Sigma(f)=j^kf^{-1}(\Sigma)$ is a \emph{semialgebraic singularity}.
\end{definition}
}
{\begin{proposition}\label{prop:semiall}
Let $M=S^d$ and let $I\subset \R$ be an interval. If $\Sigma$ is one among the sets $S^d\times I, 0_{TM}\times I,  W_{\text{mean=0}}, \{u\}' $ defined in \cref{ex:typeW}, then it is isotropic and semialgebraic.
\end{proposition}
\begin{proof}
We explain in detail the first two cases and only sketch the others.
 \begin{enumerate}[(i.)]
    \item Excursion sets and level sets and more generally, all type-$\Sigma$ singularities of degree $k=0$ of the form $\Sigma=S^d\times I$, are obviously isotropic, because $j^0f(x)=(x,f(x))$. Indeed, in this case \cref{eq:isotropR} becomes: $f(Rx)\in I\iff f\circ R(x)\in I$. Since intervals are semialgebraic, we conclude.
    \item In the case of critical points, the set $\Sigma=0_{TM}\times \R\subset J^1(M)$ is isotropic because of the chain rule: $\nabla (f\circ R)(x)=R^{T}\nabla f(Rx)$. In such case, \cref{eq:isotropR} becomes: $\nabla f(Rx)=0 \iff \nabla (f\circ R)(x)=0$, thus it holds simply because $R$ is a diffeomorphism. Since $\Sigma$ is a linear subspace of $V$, it is semialgebraic.
    \item For the other cases, the fact that the mean curvature is invariant under ambient isometries and the chain rule imply that $\Sigma$ is isotropic. The semialgebraicity of $W_{\text{mean=0}}$ is due to the fact that the formula for the mean curvature is a semialgebraic function of $j^kf(x)$. Finally, one can deduce the semialgebraicity of $W=\{u\}'$, o by observing that the subset of $\R^{2d}$ consisting of linearly dependent pairs of vectors in $\R^d$ is semialgebraic.
\end{enumerate}
\end{proof}}
{One can immediately observe that a semialgebraic singularity $Z\subset S^d$, in the sense of the above definition, is also a semialgebraic subset of the sphere. Among other things, such a property ensures that $Z$ admits a Whitney stratification, and therefore that its Lipschitz-Killing curvatures (sometimes called \emph{intrinsic volumes})
$
\{\m L_j(Z)\colon j=0,\dots,d-1\}
$
are well defined, see \cite[Definition 10.7.2]{AdlerTaylor}. We refer to \cite{AdlerTaylor} for a complete treatment of such a topic. Furthermore, we will see below that the fact that $Z$ is semialgebraic ensures that its Lipschitz-Killing curvatures are controlled by a constant depending on the degree of the polynomials used in its definition (see \cref{eq:semialgdef}). The same can be said for the Betti numbers of $Z$,
$
\{b_j(Z)\colon j=0,\dots,d-1\}
$, which satisfy Thom-Milnor's bound: a classical result of semialgebraic geometry.
\begin{lemma}\label{lem:finiteness}
    Let $\Sigma\subset J^k(S^d,\R^a)$ be closed and semialgebraic and let $f\in \mC^\infty(S^d,\R^a)$ be the restriction of a polynomial of degree $\ell$ in $\R^{d+1}$. Then, there exists a constant $C>0$ such that
    \be 
\max_{j=0,\dots,d} \max\kop \m L_j\tyu Z_\Sigma(f)\uyt, b_j\tyu Z_\Sigma(f)\uyt \pok \le C\ell^{d+1},
    \ee
    where the constant $C=C(\Sigma)$ depends only on $\Sigma$ (thus, implicitely, also on $k,d,a$) and not on the polynomial $f$.
\end{lemma}
\begin{proof}
The semialgebraic subset $Z= Z_\Sigma(f)\subset S^d$ is defined by an expression as that of \cref{eq:semialgdef}, where all the polynomials $p_{ij}$ are obtained by composing those defining $\Sigma$, with $j^kf$. The latter is still the restriction of a polynomial of degree $\ell$. Therefore, all the polynomials $p_{ij}$ have degree bounded by $C\ell$, where $C$ is a constant depending on $\Sigma$. By including in such collection, the degree two equation: $|x|^2=1$ defining the sphere, we deduce that $Z$ is a subset of $\R^{d+1}$, defined by a finite number of closed polynomial inequalities of degree at most $C\ell$.
An application of Thom-Milnor Bound (precisely, we use \cite[Theorem 3]{MilnorBound}) yields the following inequality:
\be\label{eq:ThomMilnor} 
\sum_{i=0}^{d}b_i(Z)\le c(\Sigma)\ell^{d+1},
\ee
where $c(\Sigma)>0$ is a positive constant depending on $\Sigma$. Now let $j\in \N$ and let us consider an affine function $A\colon S^d\to \R^j$, that is, a polynomial of degree $1$, and observe that 
\be 
Z_\Sigma(f)\cap A^{-1}(0)= Z_{\Sigma\oplus \{0_j\}}(f,L),
\ee
where $\Sigma\oplus \{0_j\}:=\{ j^k(F,G)(x)\in J^k(S^d,\R^{a+j}) :\ j^kF(x)\in \Sigma,\ G(x)=0\}$ is semialgebraic.
Therefore, applying \cref{eq:ThomMilnor} to the function $(f,L)\in \mC^{\infty}(S^d,\R^{a+j})$, which is the restriction of a polynomial of degree $\ell$, we get the following bound for the Euler-Poincaré characteristic $\chi$ of any affine section of $Z=Z|_\Sigma(f)$:
\be 
\chi\tyu Z\cap K \uyt = \sum_{i=0}^{d}(-1)^ib_i(Z\cap K)\le \sum_{i=0}^{d}b_i(Z\cap K)\le c(\Sigma\oplus \{0_b\})\ell^{d+1},
\ee
for any affine subspace $K=\ker L\subset \R^{d+1}$ of codimension $j$. We conclude combining the latter equation with Hadwiger's formula (see \cite[Equation (6.3.11)]{AdlerTaylor}):
\be 
\m L_j(Z)=\binom{d}{j-1}\frac{s_{d}}{s_{j-1} s_{d-j}}\int_{\R^{d+1}}\int_{O(d+1)} \chi\tyu Z\cap \tyu u+gK_{j} \uyt\uyt \dd g \dd u.
\ee
Here, $\dd g$ denotes the Haar probability measure of $O(d+1)$ and $K_j=\R^{d+1-j}\times \{0\}$. 
\end{proof}
}
\subsection{Proofs of \cref{c:splendid0} and \cref{c:splendid1}}
{We will prove something slightly more general than \cref{c:splendid0}. Let us consider a functional defined as follows
\be \label{eq:FeiZFunctionals}
E(\gamma)=F\tyu e_1\tyu Z_{\Sigma_1}(\tilde{f})\uyt,\dots, e_r\tyu Z_{\Sigma_r}(\tilde{f})\uyt\uyt,
\ee 
with $F\colon\R^r\to\R$ a locally bounded measurable function, $\Sigma_i\subset J^k(M)$ closed, isotropic and semialgebraic subsets. and $e_i\in \{\m L_0,\dots,\m L_{d-1}\}\cup \{ b_0,\dots,b_{d-1}\}$.
\begin{lemma}\label{lem:FeizFunctionals}
    $E(\gamma)$ defined as in \cref{eq:FeiZFunctionals} is in $L^2$ and $E(\gamma)[1]=E(\gamma)[2]=0$.
\end{lemma}
\begin{proof}
Due to semialgebraicity, \cref{lem:finiteness} ensures that $e_i( Z_{\Sigma_i}(\tilde{f}))$ is bounded, therefore $E(\gamma)$ is bounded, hence in $L^2$. moreover, clearly $E(t\gamma)=E(\gamma)$ for all $t>0$. Since $\Sigma_i$ is isotropic, by definition, we get that
\bega 
E(R^*\gamma) &=F\tyu\dots, e_i\tyu\kop x\in S^d: j^k(\lebf\circ R^T)(x) \in \Sigma_i \pok\uyt,\dots \uyt
\\
&=F\tyu\dots, e_i\tyu \kop x\in S^d: j^k(\lebf)(R^Tx) \in \Sigma_i \pok\uyt,\dots \uyt
\\
&=F\tyu\dots, e_i\tyu R Z_{\Sigma_i}(\tilde{f})\uyt,\dots \uyt
=E(\gamma);
\eega
where the last equality is due to the fact that $e_i(RZ)=e_i(Z)$ for any isometry $R\colon S^d\to S^d$. We conclude by an application of \cref{thm:Main1}. 
\end{proof}
\begin{proof}[Proof of \cref{c:splendid0}]
Apply the \cref{lem:FeizFunctionals} above, with $r=1$ and $F=\text{id}$. 
\end{proof}
\begin{proof}[Proof of \cref{c:splendid1}]
By \cref{prop:semiall}, we are in the position of applying \cref{c:splendid0} to all functionals $\f(\tilde{f})$ considered in the statement of \cref{c:splendid1}. Therefore, the thesis.
\end{proof}
}
\section{The evaluation of projections: a key idea}
\label{sec:magictrick}

{In the rest of the paper, we will perform more explicit computations of the Wiener chaos expansion of local functionals. In this section we discuss a key idea, that will be the starting point of all such computations, that is \cref{lem:multimagictrick}.
Notice that in \cref{thm:Main2} and \cref{t:wienerchaos}, we have to study functionals of the form 
\be 
\m E (\lebf)=\int_M F(\lebf(x)) dx,
\ee
where $F\colon \R\to \R$. By the linearity properties of Wiener chaos, one obtains the chaos expansion of $\m E(\lebf)$ simply by integrating that of $F(\lebf(x))$ over the whole manifold. This strategy is particularly convenient in our setting (that of \cref{sec:ourmodel}), given that $\lebf(x)$ has the same law for all $x\in M$, hence the coefficients of the decomposition will not depend on $x$. Moreover, if we write $\lebf(x)$ as in \cref{eq:therootofalltricks}, we can observe that 
\be\label{eq:Fleblaw} 
F(\lebf(x))\overset{\mathrm{law}}{=}E(\gamma)=e(\gamma^N,\|(\gamma^1,\dots,\gamma^{N-1})\|),
\ee
where $\gamma \sim \m N(0,\mathbbm{1}_N)$, for an appropriate function $e\colon \R\times [0,+\infty)\to \R$. So, $E(\gamma)$ belongs to the following class of functionals.
\begin{definition}\label{def:A}
Let $k\ge2$ be an integer such that $\gamma=(\gamma_1,\ldots,\gamma_k)$ for $\gamma_i\sim\mathcal N(0,\mathbbm{1}_{N_i})$ Gaussian vectors of finite dimension $N_i$, $i=1,\ldots,k$. 
Let $\mathscr{A}(N_1,\dots,N_k)\subset L^2$ be the family of functionals of the form $F(\gamma) = e(\gamma_1,\|\gamma_2\|,\ldots,\|\gamma_k\|)$, for a function $e:\R^{N_1}\times \R^{k-1} \to\R$ such that $\E[e(\gamma_1,\|\gamma_2\|,\ldots,\|\gamma_k\|)^{2}]<\infty$. Equivalently,
 \begin{multline}\label{eq:A}
        \mathscr{A}(N_1,\dots,N_k) := \{ F(\gamma) \in L^2(\gamma): F(x,D_2y_2,\ldots,D_ky_k) = F(x,y_2,\ldots, y_k) \\ \forall (x,y_2,\ldots,y_k)\in \R^{N_1}\times\R^{N_2}\times\ldots\times \R^{N_k}, \, \forall D_i\in O(N_i), \, i=2,\ldots,k \}.
    \end{multline}
\end{definition}
The main result of this section, \cref{lem:multimagictrick}, is a characterization of the chaotic expansions of functionals belonging to $\mathscr{A}(N_1,\dots,N_k)$. Precisely, we will show that $\mathscr{A}(N_1,\dots,N_k)$ is generated by the following orthonormal family. 
\begin{lemma}\label{lem:SphericalChaos} The following family is orthonormal and is contained in the $q$-th Wiener-chaos space relative to $\gamma\sim \m N(0,\mathbbm{1}_N)$,
    \begin{multline}\label{eq:sphericalChaos}
        \mathscr{H}_q(N_1,\dots,N_k):=\Bigg\{ \frac{
        H_{q_1}(\gamma_1)
        }{
        \sqrt{q_1!}
        }
        \prod_{j=2}^k\left(\int_{S^{N_j-1}} H_{q_j}(\gamma_j^{T}v) dv\right)  \frac{1}{\sqrt{q_j!\beta(N_j,q_j)s_{N_j-1}}}: 
        \\
        q_1\in \N^{N_1}, \, q_j \in 2\N \text{ for } j=2,\ldots,k : |q_1| + q_2 +\ldots + q_k = q\Bigg\};
    \end{multline}
    where $\frac{H_{q_1}(\gamma_1)}{\sqrt{q_1!}}$ is understood in the multi-index notation, see \cref{d:wiener}, and, given integers $N>1$, and $q\ge 1$, we define:
\be\label{eq:beta}
        \beta(N,q)  := \int_{S^{N-1}} |v_1|^q dv
       =s_{N-2}
       B\tyu \frac{q+1}{2},\frac{N-1}{2}\uyt
=\frac{2\pi^{\frac{N-1}2}\Gamma(\frac{q+1}{2})}{\Gamma(\frac{N+q}{2})}.
\ee
\end{lemma}
\begin{proof}
Since $\gamma_1,\dots,\gamma_k$ are independent, the elements of $\mathscr{H}(N_1,\dots,N_k)$ are obviously orthogonal and are in the $q$-th chaos. The only thing to prove is that their variance is one. By independence, it is enough to check it for each factor:
\begin{equation}\label{conto}
        \E \left[ \left( \int_{S^{N-1}} H_q(\gamma^T v) dv \right)^2\right]
          = q! \int_{S^{N-1}\times S^{N-1}} |v^T w|^q dvdw = q!\int_{S^{N-1}}
          \int_{S^{N-1}} |v_1|^q dvdw
    \end{equation}
the first identity is valid for all $q\in 2\N$, which is what we needed.
The identities in \cref{eq:beta} are proven in \cite[Lemma C.1]{cgv2025StecconiTodino}, and $B(z,w)=\frac{\Gamma(z)\Gamma(w)}{\Gamma(z+w)}$ is the classic Beta function.
\end{proof}
}

{
We will only need the case $N_1=1, k=1$, for functionals as that in \cref{eq:Fleblaw}, and we will need the case $N_1=1, k=2, q=2$, for \cref{thm:louis2} in the next section. We choose to discuss the general case, since it is not so much more complicated than the latter. 
\begin{remark}\label{rem:multrick2}
Note that the case $q=2$, there is no mixed product terms in \cref{eq:sphericalChaos}, that is, for instance, $\mathscr{H}_2(1,N)=\{\frac1{\sqrt{2}}H_2(\gamma_1), \frac{1}{\sqrt{2N}}(\|\gamma_2\|^2-N)\}$, Indeed, $\beta(N,2)=\frac{s_{N-1}}{N}$, and we have
\be 
\frac{1}{\sqrt{2\beta(N,2)s_{N-1}}}\int_{S^{N-1}} H_{2}(\gamma^{T}v) dv=\frac{1}{\sqrt{N}}\sum_{j=1}^N\frac{H_2(\gamma^{(j)})}{\sqrt{2}}=\frac{\|\gamma\|^2-N}{\sqrt{2N}}
\ee
by straighforward computations, see also \cref{lem:multimagictrick2222} below.
\end{remark}
The next Lemma implies that $\cup_{q\in \N}\mathscr H_q (0,N)$ spans the space $\mathscr{A}(0,N)$; then, one can easily obtain the more general case by induction, see \cref{lem:multimagictrick} below.
}

\begin{lemma}\label{lem:magictrick} Let $\gamma\sim\mathcal N(0,\mathbbm{1}_N)$ be a standard Gaussian vector of finite dimension $N$. Let $F:\R^{N}\to\R$ be such that $\E[F(\gamma)^{2}]<\infty$. Let us assume that $F$ is invariant up to orthogonal transformations, that is, $F(Dx) = F(x)$ for any $x\in \R^{N}$ and $D\in O(N)$\footnote{Equivalently, we may ask that $F(\gamma) = f(\|\gamma\|)$, for a function $f:\R\to\R$ such that $\E[f(\|\gamma\|)^{2}]<\infty$.}. Then,
	\begin{equation}\label{eq:Fgamma[q]}
        F(\gamma)[q] = c_N (q) \times \int_{S^{N-1}} H_q(\gamma^T v) dv,
    \end{equation}
    where \begin{equation}\label{eq:cq2}
        c_N(q) = \frac{\E [F(\gamma) H_q(\gamma^T e_1)]}{q! \beta(N,q)} \qquad \text{ if } 2|q \text{, and } \, 0 \text{ otherwise.}
    \end{equation}
    \end{lemma} 
    \begin{proof} By definition, the $q$th chaos of $\gamma$ is \begin{equation}
    C_q = \left\{ \int_{S^{N-1}} H_q(\gamma^\perp v) \mu(dv) : \text{ } \mu \text{ measure on } S^{N-1} \right\}.
\end{equation}
Then, we may write
    \begin{equation}
        F(\gamma)[q] = \int_{S^{N-1}} H_q(\gamma^\perp v) \mu(dv)
    \end{equation}
    where $\mu$ is a measure on $S^{N-1}$, defined up to orthogonal transformations $D\in O(N)$, since $F(Dx) = F(x)$ for all $x\in\mathbb R^N$. Therefore, $\mu$ is a multiple of the uniform measure on $S^{N-1}$. This proves \eqref{eq:Fgamma[q]}. The coefficient is 
    \begin{equation}\label{eq:cq1}
        c_N(q) = \frac{\E \left[F(\gamma) \int_{S^{N-1}} H_q(\gamma^\perp v) dv\right]}{\E \left[ \left( \int_{S^{N-1}} H_q(\gamma^\perp v) dv \right)^2\right]}.
    \end{equation}
    The denominator in \eqref{eq:cq1} was already computed in \cref{conto} above.
    The numerator can be simplified as well thanks to the invariance of $F$: \begin{align}
        \E \left[F(\gamma) \int_{S^{N-1}} H_q(\gamma^\perp v) dv\right] 
        & = \int_{S^{N-1}} E[F(\gamma)H_q(\gamma^\perp e_1)]dv \\
        & = s_{N-1} E[F(\gamma)H_q(\gamma^\perp e_1)].
    \end{align}
Therefore, we obtain \eqref{eq:cq2}.
\end{proof}

The previous lemma gives an alternative and more compact way to express the chaotic decomposition of a functional with rotational invariance. The following one is an extension to functionals that are invariant up to rotations of fixed subsets of variables. 


\begin{theorem}\label{lem:multimagictrick} 
{The space $\mathscr{A}(N_1,\dots,N_k)$ (of \cref{def:A}) is generated by the orthonormal system $\cup_q\mathscr{H}_q(N_1,\dots,N_k)$ (of \cref{lem:SphericalChaos}). In particular, if $F(\gamma)\in \mathscr{A}(N_1,\dots,N_k)$} then, 
    \be
F(\gamma)[q] = \sum 
c_{\underline{q}} \, H_{q_1}(\gamma_1) \, \prod_{j=2}^k \tyu \int_{S^{N_j-1}} H_{q_j}(\gamma_j^{T}v) dv\uyt,
    \ee
    where the sum ranges over multi-indices $\underline{q}\in \N^{N_1}\times 2\N^{k-1}$, such that $|\underline{q}|=|q_1|+q_2+\ldots+q_k=q$, (note that $q_1\in \N^{N_1}$ is a multi-index itself) and the coefficients are \be
    c_{\underline{q}} = \frac{\E[ F(\gamma) H_{q_1}(\gamma_1) \prod_{j=2}^k H_{q_j}(\gamma_j^{(1)})]}{\underline{q}! \prod_{j=2}^k\beta(N_j,q_j)} = \frac{\E[ F(\gamma) H_{q_1}(\gamma_1) H_{q^\prime}(\gamma^\prime)]}{\underline{q}! \prod_{j=2}^k\beta(N_j,q_j)},
    \ee
    where $q^\prime = (q_2,\ldots,q_k)$ and $\gamma^\prime = (\gamma_2^{(1)},\ldots,\gamma_k^{(1)})$ is the Gaussian vector of the first coordinates of $\gamma_2, \ldots, \gamma_k$. 
\end{theorem}
\begin{proof}
    If $k=2$, we may write \begin{align}
        F(\gamma)[q] & = \sum_{\substack{\underline{q} = (q_1,q_2)\in \N^{N_1}\times \N^{N_2}:\\|q_1|+|q_2|=q}} \frac{\E[F(\gamma) H_{q_1}(\gamma_1)H_{q_2}(\gamma_2)]}{q_1!q_2!} H_{q_1}(\gamma_1)H_{q_2}(\gamma_2) \\
        & = \sum_{q_1\in \N^{N_1}} H_{q_1}(\gamma_1) \frac1{q_1!} \sum_{\substack{q_2\in \N^{N_2}:\\|q_2|=q-|q_1|}} \frac1{q_2!} \E[ \widetilde F_{q_1}(\gamma_2) H_{q_2}(\gamma_2) ] H_{q_2}(\gamma_2) \\
        & = \sum_{q_1\in \N^{N_1}} H_{q_1}(\gamma_1) \frac1{q_1!} c_{N_2}(q-|q_1|) \int_{S^{N_2-1}} H_{q-|q_1|} (\gamma_2^Tv) dv,
    \end{align}
    where $\widetilde F_{q_1}(\gamma_2) = \E[ F(\gamma) H_{q_1}(\gamma_1) | \gamma_2]$ is a function of $\gamma_2$ (by independence) that is invariant under rotations $D\in O(N_2)$, since $F$ is. Then, the last equality follows by applying \cref{lem:magictrick}, and \begin{align}
        c_{N_2}(q-|q_1|) = \frac{\E[ F(\gamma) H_{q_1}(\gamma_1) H_{q-|q_1|} (\gamma_2^{(1)})]}{q_1!(q-|q_1|)! \beta(N_2,q-|q_1|)}
    \end{align}
    if $q-|q_1|$ is even, and $0$ otherwise. If $k\ge 3$, we may argue by induction. Reasoning as in the case $k=2$, but summing from the $k$th term first, we may write \begin{align}
     F(\gamma)[q] & = \sum_{q_k\in 2\N} \frac1{q_k!\beta(N_k,q_k)} \int_{S^{N_k-1}} H_{q_k}(\gamma_k^Tv) dv \times \\
        & \qquad \times \sum_{\substack{q^\prime\in \N^{N_1+\ldots + N_{k-1}}:\\|q^\prime|=q-q_k}} \frac{\E[F(\gamma) H_{q^\prime} (\gamma_1,\ldots,\gamma_{k-1}) H_{q_k}(\gamma_k^{(1)}) ]}{q^\prime!} H_{q^\prime} (\gamma_1,\ldots,\gamma_{k-1}).
    \end{align}
    Remark that the second sum gives the $(q-q_k)$th chaotic component of $\widetilde F(\gamma_1,\ldots,\gamma_{k-1}):= \E[ F(\gamma)H_{q_k}(\gamma_k^{(1)})|\gamma_1,\ldots,\gamma_{k-1}]$. Hence, by inductive hypothesis, we may write \begin{align}
        F(\gamma)[q] & = \sum_{q_k\in 2\N} \frac1{q_k!\beta(N_k,q_k)} \int_{S^{N_k-1}} H_{q_k}(\gamma_k^Tv) dv \times \\
        & \qquad \times \sum_{\substack{q^\prime = |q_1|+q_2+\ldots+q_{k-1}:\\
        q_1\in \N^{N_1}, q_2,\ldots,q_{k-1}\in 2\N, \\
        |q^\prime| =q-q_k}} c_{q^\prime} H_{q_1}(\gamma_1) \prod_{j=2}^{k-1} \left(\int_{S^{N_j-1}} H_{q_j}(\gamma_j^Tv) dv \right),
    \end{align} 
    where \begin{align}
        c_{q^\prime} & = \frac{\E\big[ \E[ F(\gamma)H_{q_k}(\gamma_k^{(1)})|\gamma_1,\ldots,\gamma_{k-1}] H_{q_1}(\gamma_1) \prod_{j=2}^{k-1} H_{q_j}(\gamma_j^{(1)})\big]}{q^\prime! \prod_{j=2}^{k-1}\beta(N_j,q_j)}\\
        & = \frac{\E\big[  F(\gamma)H_{q_1}(\gamma_1) \prod_{j=2}^{k} H_{q_j}(\gamma_j^{(1)})\big]}{q^\prime! \prod_{j=2}^{k-1} \beta(N_j,q_j)}.
    \end{align}
    This is enough to conclude.
    \end{proof}
If $q=2$, \cref{lem:multimagictrick} simplifies as follows.
\begin{corollary}\label{lem:multimagictrick2222} 
Let $F(\gamma)\in \mathscr{A}(N_1,\dots, N_k)$, see \cref{def:A}. Then,  
\be
F(\gamma)[2] = \sum_{j=1}^{N_1} c_1^{(j)} H_2(\gamma_1^{(j)}) +  \sum_{\substack{j,l=1,\ldots,N_1\\j\neq l}} c_1^{(j,h)} \gamma_1^{(j)}\gamma_1^{(h)}
+ \sum_{i=2}^k c_i \left(\|\gamma_i\|^2 - N_i\right) ,
    \ee
    where $\gamma_1^{(j)}$ denotes the $j$th component of the vector $\gamma_1$, for $j=1,\ldots, N_1$, and \bega
    c_{1}^{(j)}  & = \frac12\E [F(\gamma) H_2(\gamma_1^{(j)})],\quad 
    c_{1}^{(j,h)}  = \E [F(\gamma) \gamma_1^{(j)}\gamma_1^{(h)}], \qquad && j,h = 1,\ldots,N_1,\\
    c_{i} & = 
    \frac{\E [F(\gamma) (\|\gamma_i\|^2-N_i)]}{2N_i}, \qquad && i =2,\ldots,k.
    \eega
    In particular, if $N_1=1$, we have \be
    F(\gamma)[2] = c_1  (\gamma_1^2 -1) + \sum_{i=2}^k c_i \left(\|\gamma_i\|^2 - N_i\right).
    \ee
\end{corollary} 
\begin{proof}
The statement follows immediately by specializing \cref{lem:multimagictrick} to the case $q=2$ and recalling \cref{rem:multrick2}.
\end{proof}

\section{Main Result II: Formal statements and proofs}\label{sec:main2}
{
\subsection{A cancellation theorem for local functionals: Proof of 
\cref{thm:Main2}}
\begin{proof}[Proof of \cref{thm:Main2}]
Recall (\cref{eq:therootofalltricks}) that 
\be\label{eq:rrrr} 
\tilde{f}(x)=\frac{c f(x)}{\sqrt{f(x)^2+\|\gamma_x\|^2}},\quad \text{where}\quad c:=\sqrt{\frac{N}{\vol{} (M)}},
\ee
and that $\gamma_x$ is the orthogonal projection in $\R^N$ of $\gamma$, onto the space $\y(x)^\perp$, so that $\gamma=\gamma_x+f(x)\y(x)$. This implies that, for every $x\in M$, $\tilde{f}(x)$ is distributed as $c\frac{\gamma^N}{\|\gamma\|}$, which implies that $I(\tilde{f})$ is square-integrable, under the assumptions of the theorem. 
Observe that since the functional $F(c\frac{\gamma^N}{\|\gamma\|})$ depends on $(\gamma^1,\dots,\gamma^{N-1})$ only via its norm, its chaotic decomposition must have the following form (see \cref{lem:multimagictrick}):
\bega
F\tyu c\frac{\gamma^1}{\|\gamma\|}\uyt [2]&=c_1H_2(\gamma^1)+c_2 \sum_{j=1}^{N-1}H_2(\gamma^j)=(c_1-c_2)(\gamma^1)^2+c_2\|\gamma\|^2-(c_1+(N-1)c_2)
\eega
Therefore, by integrating in $M$, we obtain
\bega 
I(\tilde{f})[2]&=\int_M F\tyu \tilde{f}(x)\uyt [2]
=\int_M (c_1-c_2)f(x)^2+c_2 \|\gamma\|^2-(c_1+(N-1)c_2)dx
\\&= \frac{c_1+(N-1)c_2}{c^2}(\|\gamma\|^2-N),
\eega
where we only used that $c\|f\|_{L^2}=\|\gamma\|$ and that $c^2\vol{}(M)=N$, because of \cref{eq:rrrr}. To conclude, observe that 
\bega 
c_1+(N-1)c_2&=\E\kop F\tyu c\frac{\gamma^1}{\|\gamma\|}\uyt \frac{ H_2(\gamma^1)}{2}\pok+(N-1)\E\kop F\tyu c\frac{\gamma^1}{\|\gamma\|}\uyt \frac{ \sum_{j=1}^{N-1}H_2(\gamma^j)}{2(N-1)}\pok
\\
&=
\frac12\E\kop F\tyu c\frac{\gamma^1}{\|\gamma\|}\uyt (\|\gamma\|^2-N)\pok=0.
\eega
The last term is zero because $\frac{\gamma^1}{\|\gamma\|}$ and $\|\gamma\|$ are independent.
\end{proof}
}

{\subsection{Another cancellation theorem, for local functionals of a field and its derivatives}
In this section we prove the results discussed at the end of \cref{subsec:mainres2}. First of all let us clarify how the Homotheticity Assumption: \cref{eq:homothetic} relates to the notion of \emph{homothetic field}, given in \cite[Definition 1.11]{cgv2025StecconiTodino}. To this end, let us recall that $\E \{f(x)\}=1$ for all $x\in M$, see \cref{rem:varVSvol}. This implies that the \emph{average frequency} $\lambda(f)$ of $f$, as defined in \cite[Definition 1.11]{cgv2025StecconiTodino}, is
\be \label{eq:avgfreq}
\textbf{Average frequency:} \quad \lambda(f)=\sqrt{ \fint_M \E\kop \|\nabla f(x)\|^2 \pok dx }\overset{*}{=}\sqrt{ -\E\fint_M f(x)\Delta f(x) dx }
\ee
We recall from \cite[Section 2.0.1]{cgv2025StecconiTodino}, that such terminology is motivated by the fact that, due to Green's formula ($\overset{*}{=}$) above, if $f$ is a random wave with interval window $I$, then $\lambda(f)^2$ is the average among the eigenvalues in $I$ (counted with multiplicity).
\begin{lemma}\label{lem:homothe}
Let the setting of \cref{sec:ourmodel} prevail. The following two conditions are equivalent: 
\begin{enumerate}[$(i)$]
\item Homotheticity Assumption (\cref{eq:homothetic}): There exists a constant $\xi>0$ such that:
\be \label{eq:homothetic}
\|\de_uY(x)\|= \xi \|u\| \quad \forall x\in M,\ \forall u\in T_xM.
\ee 
    \item The field $f$ is \emph{homothetic} (as in \cite[Definition 1.11]{cgv2025StecconiTodino}): 
\be 
 \E\kop |\de_uf(x)|^2\pok = \frac{\lambda(f)^2}{d}\|u\|^2 \quad \quad \forall x\in M,\ \forall u\in T_xM.
 \ee
\end{enumerate}
Moreover, in case they hold, we have $\xi=\lambda(f)\sqrt{\frac{N}{d\vol{}(M)}}$.
\end{lemma}
\begin{proof}
Recall that $f(x)=\langle \gamma , \y(x) \rangle$. The equivalence follows from the following identities:
\be
\E\kop |\de_uf(x)|^2\pok=\E\kop |\langle \gamma, d_u\y(x)\rangle |^2\pok=\|\de_u\y(x)\|^2=\|\de_uY(x)\|^2 \frac{\vol{}(M)}{N};
\ee
all such identities are straightforward, given the setting of \cref{sec:ourmodel}. Moreover, \cref{eq:avgfreq} can be equivalently written as follows:
\be 
\frac{\lambda(f)^2}d=\fint_M \fint_{S(T_xM)}\E\kop |\de_u f(x)|^2 \pok 
=
\frac{\vol{}(M)}{N}\fint_M \fint_{S(T_xM)} \|\de_uY(x)\|^2 dx .
\ee
This proves the last statement and concludes the proof of the lemma.
\end{proof}
}
{In the next statement, we use the notation $c=\sqrt{\frac{N}{\vol{}(M)}}$ as in \cref{eq:therootofalltricks}}.
{
\begin{theorem}\label{thm:louis2}
Let the setting of \cref{sec:ourmodel} prevail and assume, in addition, that the Homotheticity Assumption (\cref{eq:homothetic}) holds. Let $\lambda=\lambda(f)$ be the average frequency of $f$, defined as in \cref{eq:avgfreq}.
Let $F\colon \R\times [0+\infty)\to \R$ be measurable and such that $F\tyu c\frac{\gamma^N}{\|\gamma\|},\lambda^{-1}\frac{\|(\gamma^1,\dots,\gamma^d)\|}{\|\gamma\|}\uyt$ is square-integrable. Then, the functional 
\be 
\m E(\tilde{f})=\int_M F\tyu \tilde{f}(x), \|\nabla\tilde{f}(x) \|\uyt dx,
\ee
is square-integrable and there are constants $c_1,c_2,c_3$ depending only on $F, \vol{}(M)$ and $N$, such that:
\bega
\m E(\tilde{f})[1]&= c_1 \int_M f(x) dx,
\\
\m E(\tilde{f})[2] &= c_2 \int_M (f(x)^2-1) dx + c_3\int_M f(x)(\Delta f(x)-\lambda)dx.
\eega
Moreover, if $f$ is also a random wave on $M$ with interval window $I$ (see \cref{d:whatisarandomwave}), then 
\bega 
\m E(\tilde{f})[1]&= c_1 
\frac{1}{c}\delta_0(I) \xi^0,
\\
\m E(\tilde{f})[2] &= \frac{c_3}{c^2}\sum_{\{i\in \N:\lambda_i\in I\}}((\xi^i)^2-1)(\lambda^2-\lambda_i^2),
\eega
{where $\xi^i\sim \m N(0,1)$ is defined as in \cref{eq:relabel} and }where, as before, $c_3$ does not depend on the choice of $I$.
{As a consequence, in this case one has that, either $c_3 = 0$ and $\m E(\tilde{f})[2] = 0$ for all intervals $I$, or $c_3\neq 0$ and $\m E(\tilde{f})[2] = 0$ if and only if $f$ is purely monochromatic (that is, if $I=\{\lambda\}$). }
\end{theorem}}
{\begin{proof}
We follow the same scheme as that of the previous proof, but we refine the argument to include the derivative. Let us fix $x\in M$, and let us fix an orthonormal basis $\de_1(x)\dots,\de_d(x)$ of $T_xM$ 
The Constant Norm Assumption (\cref{eq:YlmNorm}) and the homotheticity assumption (\cref{eq:homothetic}) ensure that the vectors $\y(x)=\frac{Y(x)}{c},\xi^{-1}\de_1 Y(x),\dots,\xi^{-1}\de_dY(x)$ are orthonormal in $\R^N$, generating a space that we will denote as $A_x\subset \R^N$. Therefore, recalling that $f(x)=\langle \gamma, \y(x)\rangle$, we can easily see that
\bega 
\gamma^\nabla_x
&:=\gamma-\tyu \langle \gamma,\y(x)\rangle \y(x) + \sum_{i=1}^d\langle \gamma,\frac{\de_iY(x)}{\xi}\rangle \frac{\de_iY(x)}{\xi} \uyt
\\
&=\gamma-\tyu f(x)\y(x) + \sum_{i=1}^d\tyu \frac{c}{\xi}\de_if(x) \uyt \xi^{-1}\de_iY(x)\uyt
\eega
is the orthogonal projection of $\gamma$ onto $A_x^\perp$. As a consequence, it is a standard Gaussian vector of rank $N-d-1$, supported on $A_x^\perp$; and $(f(x),\frac{c}{\xi}\de_1f(x),\dots,\frac{c}{\xi}\de_df(x))$ is a standard Gaussian vector in $\R^{d+1}$. Recalling \cref{eq:therootofalltricks}, we deduce that $F\tyu \tilde{f}(x), \|\nabla\tilde{f}(x) \|\uyt$ has the same distribution as the following random variable
\be 
X=F\tyu c \frac{\gamma^N}{\|\gamma\|}, \frac{c}{\xi}\frac{\|(\gamma^1,\dots,\gamma^d)\|}{\|\gamma\|}\uyt,
\ee
so that $E(\tilde{f}(x))$ is square-integrable. Given the form of the functional $X$, its first and second chaotic decompositions must have the following form (see \cref{lem:multimagictrick}):
\bega 
X[1]&=a \gamma^N
\\
X[2]&=c_0 H_2(\gamma^N)+c_1 \sum_{i=1}^{d}H_2(\gamma_i)+c_2 \sum_{i=d+1}^{N-1}H_2(\gamma_i)
\\
&=(c_0-c_2)( (\gamma^N)^2-1)+(c_1-c_2)\tyu \|(\gamma^1,\dots,\gamma^d)\|^2-d\uyt+c_2 \tyu \|\gamma\|^2-N\uyt
\eega
Threfore, by integrating over $M$, we obtain $\m E(\lebf)[1]=a\int_Mf(x) dx$ and
\bega 
\m E(\tilde{f})[2]&=(c_0-c_2)\int_M \tyu f(x)^2-1 \uyt dx + (c_1-c_2)\int_M \tyu \frac{c^2}{\xi^2}\|\nabla f(x)\|^2-d\uyt  dx+\frac{Nc_2}{c^2}(\|\gamma\|^2-N)
\\
&=c_3\tyu\int_M (f(x)^2-1)dx\uyt + c_4\int_M (\|\nabla f(x)\|^2-d)
\\
&=c_5(\|\gamma\|^2-N) + c_6\int_M f(\Delta f+\lambda^2 f),
\eega
for appropriately chosen constants $c_j\in \R$ and $\lambda>0$.  In particular, $\lambda$ is defined by the fact that the above expression must have zero expectation, being an element of the second chaos, which means $\lambda^2=-\fint_M\E f(x)\Delta f(x) dx$. The last identity is due to Green's equation: $\int_M \|\nabla f\|^2 + f(x) \Delta f(x) dx=0$. This concludes the first part of the proof. Now, let us assume that $f$ is also a random wave in the sense of \cref{d:whatisarandomwave}. It follows that $cf=\sum_{i\in \N:\lambda_i\in I}\xi^iY_i$ and that the functions $Y_i$ (besides being orthonormal due to \cref{eq:Yorthogonality}) are eigenfunctions: $\Delta Y_i=\lambda_i Y_i$, so that the expression above becomes
\be 
\m E(\tilde{f})[2]=\frac{1}{c^2}\sum_{i\in \N:\lambda_i\in I}(\xi_i^2-1)\qwe c_5+c_6(\lambda^2-\lambda_i^2) \ewq;
\ee
where since the expectation, again, should be zero, we have: $\lambda^2=\sum_{i\in \N:\lambda_i\in I}\lambda_j^2$.
At this point we can argue as in the previous proof (proof of \cref{thm:Main2}) to deduce that since $\tilde{f}(x)$ and $\|\gamma\|^2=\sum_{i\in \N:\lambda_i\in I}\xi_j^2$ are independent, we necessarily have the following:
\bega
0=\E\kop \m E(\tilde{f}(x))(\|\gamma\|^2-N)\pok=\E\kop \frac{1}{c^2}\sum_{i,j\in \N:\lambda_i,\lambda_j\in I}H_2(\xi_j)H_2(\xi_i)\qwe c_5+c_6(\lambda^2-\lambda_i^2) \ewq \pok=2Nc_5.
\eega
To conclude the proof it is sufficient to observe that the random variable $\sum_{i\in \N:\lambda_i\in I}(\xi_i^2-1)(\lambda^2-\lambda_i^2)$ is identically zero if and only if $ I=\{\lambda\}$. {(We note that the constants $c_1,c_2,c_3$ in the statement do not correspond to those is the proof.)}
\end{proof}
\begin{remark}
Consider \cref{r:afterRW}, Point \ref{itm:HomogeneouSpaz}. The Homotheticity Assumption restricts even more the class of manifolds for which we can apply \cref{thm:louis2} to the associated random waves. A natural example is given by those Riemannian homogeneous spaces $M=G/K$ for which the group acts transitively on the tangent space $TM$. Indeed this implies that the field $(x,u)\mapsto \de_u f(x)$ must have constant variance, that is precisely what \cref{eq:homothetic} says. Thus, \cref{thm:louis2} applies to the torus $M=\mathbb T^d$ and its arithmetic random waves.
\end{remark}
}

\section{Main Results III: Formal Statements and proofs}\label{sec:main3}

In this Section we focus our attention on the simplest geometric functional, the excursion area. 

\subsection{Excursion Area}
We are interested in the functional $\m E_u(\tilde f)$, defined for
$|u|\in \left[0,\frac{N}{\vol{}(M)}\right]$, defined in \cref{eq:ExcursionArea} as:

\be \label{eq:Eu}
\m E_u(\tilde f):=\int_M \I{\kop \lebf(x) \ge u\pok} \dd x.
\ee
\begin{remark}
Note that if $u > \sqrt{\frac{N}{\vol{}(M)}}$, then $\m E_u(\tilde f)=0$, by Cauchy-Schwarz, since $\lebf(x)=\langle \frac{\gamma}{\|\gamma\|},Y(x)\rangle$ and $\|Y(x)\|^2=\frac{N}{\vol{}(M)}$. For the same reason, if $u <-\sqrt{\frac{N}{\vol{}(M)}}$, then $\m E_u(\tilde f)=\vol{}(M)$.
\end{remark}

\begin{lemma}\label{lem:IFGammax} Let the setting of \cref{sec:ourmodel} prevail. Let $\y$ be as in \cref{e:unitvector}. We have that
\bega 
\I{\kop \lebf(x) \ge
 u \pok}&= 
\I{\kop \langle {\gamma}, \y(x)\rangle \ge \|\gamma_x\| \frac{u}{\sqrt{\frac{N}{\vol{}(M)}-u^2}} \pok}.
\eega
\end{lemma}
\begin{proof} 
Straightforward consequence of \cref{eq:therootofalltricks}.

\end{proof}

\subsection{Proof of \cref{thm:Main3}}

\begin{proof}
Thanks to \cref{lem:IFGammax}, we can equivalently consider the inequality: \begin{equation}
    \langle {\gamma}, \y(x)\rangle \ge \|\gamma_x\| \frac{u}{\sqrt{\frac{N}{\vol{}(M)}-u^2}},
\end{equation}
whose left and right hand side are independent. Recall the classic $L^2$ expansion for the excursion area of a Gaussian random variable $\gamma^0\sim \m N(0,1)$ above a fixed threshold $\tilde u$:
\be\label{eq:ch} 
\I{\kop \gamma^0 \ge
 \tilde u \pok}=\sum_{q\in \N} J_q(\tilde u)\frac{H_{q}(\gamma^0)}{q!},
\ee
where $J_q(u)=(-1)^{q-1}H_{q-1}(u)\frac{e^{-\frac{u^2}{2}}}{\sqrt{2\pi}}$, for $q\ge 1$, while $J_0(u)=\Phi(u)$. Choosing $\gamma^0 = \langle {\gamma},\y(x)\rangle$, and the conditional threshold $\tilde u = \|\gamma_x\| \frac{u}{\sqrt{\frac{N}{\vol{}(M)}-u^2}}$, the decomposition follows. Let us now investigate the asymptotics behaviour of the coefficients $J_q$, for a fixed $q\in\N$. Since $\|\gamma_x\|\sim \chi_{N-1}$, a chi distribution with $N-1$ degrees of freedom, we have that \bega 
& \E\tyu \|\gamma_x\| \frac{u}{\sqrt{\NN-u^2}}\uyt= \frac{u\E(\chi_{N-1})}{\sqrt{\NN-u^2}} \xrightarrow[N\to \infty]{} u\sqrt{\vol{}(M)},\\
& \Var \tyu \|\gamma_x\| \frac{u}{\sqrt{\NN-u^2}}\uyt = \frac{u^2\Var(\chi_{N-1})}{\NN-u^2} \xrightarrow[N\to \infty]{} 0,
\eega
uniformly in $x$ for any fixed value of $u>0$. Since $J_q(u)\le C_q|u|^{q-1}$ for a positive constant $C_q$, it is enough to check that 
\be
\E\qwe\tyu \|\gamma_x\| \frac{u}{\sqrt{\NN-u^2}}\uyt^q\ewq= \tyu \frac{u}{\sqrt{\NN-u^2}} \uyt^q\E(\chi_{N-1}^q)\le C^\prime_q,
\ee
for all $q\in \N$, uniformly in $N$, and a positive constant $C_q^\prime$. Choosing an even integer $q=2m$, for some $m\in\N$, and $N-1=2\ell$, for some $\ell\in\N$, we may write
\bega 
\tyu \frac{u}{\sqrt{\NN-u^2}} \uyt^q\E(\chi_{N-1}^q)
&=
\tyu \frac{u}{\sqrt{\frac{2\ell+1}{\vol{}(M)}-u^2}} \uyt^{2m}2^m\frac{(m+\ell-1)!}{(\ell-1)!}
\\
&\le C^\prime_m\frac{1}{\ell^m}\ell^{m}\le C_m^\prime,
\eega
allowing us to conclude.
\end{proof} 

\begin{remark}
We leave as an issue for future research to verify whether the series in $q$ converges as $N\to +\infty$, and therefore whether the asymptotic variance of the Uniform model corresponds to the Gaussian one, after neglecting the contribution from the second-order chaotic component.
\end{remark}

\subsection{Projections onto chaoses: proof of \cref{t:wienerchaos}} 
\begin{proof} Thanks to \cref{thm:Main3}, we have that $\I{\kop \lebf(x) \ge
 u \pok}$ is a functional of $(f(x),\|\gamma_x\|)$. Remark that $\gamma_x$ is a rank $N-1$ standard Gaussian vector in $Y(x)^\perp$. Therefore, $\I{\kop \lebf(x) \ge
 u \pok}\in\mathscr{A}(1,N-1)$, as in \cref{def:A}, and \cref{lem:multimagictrick} implies that 
 \begin{equation}
 \I{\kop \lebf(x) \ge
 u \pok} = \sum_{q_1\in \N} \sum_{\substack{q_2\in2\N :\\q_1+q_2=q}} 
c_{(q_1,q_2)} \, H_{q_1}(f(x)) \, \tyu \int_{S(Y(x)^\perp)} H_{q_2}(\gamma_x^{T}v) dv\uyt,
 \end{equation}
 where the coefficient is \begin{align}
 c_{(q_1,q_2)} = \frac{\E\big[ \I{\kop \lebf(x) \ge
 u \pok} H_{q_1}(f(x)) H_{q_2}(\gamma_x^{(1)})\big]}{q_1!q_2!\beta(N-1,q_2)} = \frac{\E\bigg[ J_{q_1} \tyu \|\gamma_x\|  \frac{u}{\sqrt{\frac{N}{\vol{}(N)} - u ^2}}  \uyt H_{q_2}(\gamma_x^{(1)})\bigg]}{q_2!s_{N-3}B\tyu \frac{q+1}{2},\frac{N-2}{2}\uyt}.
 \end{align}
 The last equality follows by applying \cref{thm:Main3}, and \cref{eq:beta}. Since $\gamma_x\sim \|\eta\|$, for $\eta\sim \m N(0,1_{N-1})$, we can define \begin{equation}
 \mathfrak C_N(q_1,q_2,u) = 
  \sqrt{\frac{{s_{N-2}}}{{q_1!q_2!s_{N-3}B\tyu \frac{q+1}{2},\frac{N-2}{2}\uyt}}} \E\bigg[ J_{q_1} \tyu \|\eta\|  \frac{u}{\sqrt{\frac{N}{\vol{}(N)} - u ^2}}  \uyt H_{q_2}(\eta^{(1)})\bigg].
 \end{equation}
 This is enough to conclude, since (recall also \cref{rem:multrick2}) \begin{equation}
     \E\bigg[ J_{q_1} \tyu \|\eta\|  \frac{u}{\sqrt{\frac{N}{\vol{}(N)} - u ^2}}  \uyt H_{q_2}(\eta^{(1)})\bigg] = \E\bigg[ J_{q_1} \tyu \|\eta\|  \frac{u}{\sqrt{\frac{N}{\vol{}(N)} - u ^2}}  \uyt \frac{(\|\eta\|^2 - (N-1))}{N-1} \bigg].
 \end{equation} 
 \end{proof}

\subsection{An explicit computation for the second chaos}

We already know from \cref{thm:Main1} and \cref{thm:Main2} that the second chaos of the excursion volume of a uniform random wave is zero, provided that we are in the setting of \cref{sec:ourmodel}.
However, we believe it is useful to provide an explicit, analytic derivation of its expression, and to show that its variance is zero. The reasons are two. First, our method provides an explicit expression and variance asymptotics for the second chaos of \begin{equation}
    \m E_u(\tilde f|_{A}) = \int_A \I{\kop \lebf(x) \ge u\pok}\, \dd x,
\end{equation}
where $A\subset M$ is measurable. Second, the proof of the following proposition contains all the ideas necessary for the proof of \cref{p:lowerbound}, but with a significantly smaller amount of technicalities. We note that some of the computations required for the proof are presented in the Appendix, see \cref{thm:cov2nd}.

Let us define the constant $\sigma_N:=\left(1+\frac{u^2}{\NN-u^2}\right)^{-1/2}$.

\begin{proposition}
\label{prop:secondchaos}
Let the setting of \cref{sec:ourmodel} prevail. Let $A\subset M$ be measurable. Then,
   \begin{align} \label{eq:secondchaos}
    \m E_u(\lebf|_{A})[2]  & = C_N \left( \|f|_A\|_{L^2(M)}^2  - \frac1{N-1} \int_{A} \|\gamma_x \|^2  dx \right),
\end{align}
where, denoting $v=u\sqrt{\vol{}(M)}$, \begin{align}\label{eq:coeff2}
     C_N = C_N(u) & := - \frac12 \frac{v\sigma_N^{N}}{\sqrt{2\pi}} \frac1{\sqrt{N-v^2}}\frac{\sqrt2\Gamma(N/2)}{\Gamma((N-1)/2)}  \xrightarrow{N\to\infty}  J_2 (v) ;
\end{align}
and
\begin{align}
    \Var \left( \m E_u(\lebf|_A)[2] \right) 
    & =  C_N^2 \frac{N}{(N-1)^2} \Bigg[  \Var \left( \|f|_A\|_{L^2(M)}^2 \right) \cdot N   - 2 \vol{}(A)^2 \Bigg] .
    \label{eq:BerrysCanc}
\end{align}
Moreover, if $A=M$, then the above expression is zero (i.e., $\m E_u(\lebf)[2]=0$).
\end{proposition}
{The last sentence of the above proposition can also be inferred from \cref{thm:Main2}. However, notice that in general, the restriction $f|_A$ to a subset $A$ does not satisfy the Orthogonality Assumptions \cref{eq:Yorthogonality}. That is why, for general $A$, the second chaos $\m E_u(\lebf|_A)[2]$ does not need to vanish.}

\begin{proof}
Thanks to \cref{lem:multimagictrick2222}, we have that \be
    \I{\kop \lebf(x) \ge u\pok}[2] = c_1  (f(x)^2 -1) + c_2 \left(\|\gamma_x\|^2 - (N-1)\right),
    \ee
where
\bega
    c_{1}  & = \frac12\E [\I{\kop \lebf(x) \ge u\pok} H_2(f(x))] = \frac12\E\kop J_2
 \tyu \|\gamma_x\| \frac{u}{\sqrt{\NN-u^2}}  \uyt \pok,\quad \\
    c_{2} & = 
    \frac{\E [\I{\kop \lebf(x) \ge u\pok} (\|\gamma_x\|^2-(N-1))]}{2(N-1)} = \frac12 \E\Bigg[
    \Phi
 \underbrace{\tyu \|\gamma_x\| \frac{u}{\sqrt{\NN-u^2}}\uyt}_{=:\alpha(\gamma_x)}   H_2(\gamma_x^{(1)})
 \Bigg],
    \eega
where $\Phi$ denotes the cumulative probability function of a standard Gaussian random variable. First, we compute $c_2$. Let us denote by $\f_{N}$ the density of a standard Gaussian random vector $\gamma\sim \mathcal N(0,1_{N})$. Recalling \cref{def:gammax}, let $\gamma\sim\mathcal N(0,1_{N-1})$, and let $\gamma_1$ denote its first component. By \cref{lem:gaussianCOV} (stated below), we have that
    \begin{align}
        c_2 = \frac12 \E\Bigg[
    \Phi
 {\tyu \alpha(\gamma)\uyt}   H_2(\gamma_1)
 \Bigg] & = \frac12 \E \Bigg[ \varphi_1(\alpha(\gamma))\   \frac{u }{\sqrt{\NN-u^2}} \frac{\gamma_1^2}{\|\gamma\|} \Bigg] \\  
 & = \frac12 \frac{u \ \sigma_N^{N}}{\sqrt{2\pi}\sqrt{\NN-u^2}}   \E\left[ \frac{Y^2}{\sqrt{Y^2+|Z|^2}}\right]\label{eq:uno}
    \end{align}
    where $Y\sim\mathcal N(0,1)$ and $Z\sim\mathcal N(0,1_{N-2})$ are independent. Observe that \begin{align}
        \E\left[ \frac{Y^2}{\sqrt{Y^2+|Z|^2}}\right] = \E\left[ \frac{Y^2}{Y^2+|Z|^2}\right] \E \left[ \sqrt{Y^2+|Z|^2}\right] & 
        = (N-1)^{-1} \sqrt{2}\frac{\Gamma(N/2)}{\Gamma((N-1)/2)},
    \end{align}
    since $\sqrt{Y^2+|Z|^2}\sim \chi_{N-1}$, a $\chi$ random variable with $N-1$ degrees of freedom. Hence, \begin{equation}
 c_2 =\frac12 \frac{u}{\sqrt{\NN-u^2}}  \frac{\sigma_N^{N}}{\sqrt{2\pi}} (N-1)^{-1} \sqrt{2}\frac{\Gamma(N/2)}{\Gamma((N-1)/2)} = -C_N(u) (N-1)^{-1}.
    \end{equation}
Now, we compute $c_1$. Applying \cref{lem:gaussianCOV}, we have
    \begin{equation}\label{eq:due}
        c_1=\E\kop J_2
 \tyu \|\gamma_x\| \frac{u}{\sqrt{\NN-u^2}}  \uyt \pok = \frac{\sigma_N^{N-1}}{\sqrt{2\pi}} \E[\alpha(\sigma_N\gamma)] = C_N(u). 
    \end{equation}
    Note that $\sigma_N^N = \left(1+\frac{u^2}{\NN-u^2}\right)^{-N/2} 
\xrightarrow{N\to\infty}e^{-\vol{}(M)u^2/2}$, and, by the Stirling approximation of the Gamma function, $\E [\chi_{N-1}] \sim \sqrt{N-1}$ as $N\to\infty$. This implies $C_N(u)\xrightarrow{N\to\infty} J_2(u\sqrt{\vol{}(M)})$, consistently with \cref{thm:Main3}. As a consequence, \begin{equation}\label{eq:bbr}
    \I{\kop \lebf(x) \ge u\pok}[2] = C_N(u) H_2(f(x)) - C_N(u) \frac{\|\gamma_x\|^2-(N-1)}{N-1},
\end{equation}
and, integrating over $A\subset M$, we deduce the first statement in the proof. We can now compute the variance. Recall that, equivalently to \eqref{eq:bbr} (\cref{rem:multrick2}), we have \begin{align}
    \I{\kop \lebf(x) \ge u\pok}[2] 
    & = C_N \left( H_2(f(x))  - \fint_{S^{N-2}} H_2(\gamma_x^\perp v) dv \right).
\end{align}
Integrating over $M$, computing the variance by doubling the integrals, and applying  \cref{thm:cov2nd}, we write \begin{align}
    \Var \left(\m E_u(\lebf|_{A})[2] \right) 
    & = \frac{C_N^2}4 \Bigg[ \Var \left( \int_A H_2(f(x)) dx \right) \left(1+\frac1{(N-1)^2}+\frac{2}{N-1}\right)  + \\
    & \qquad + \vol{}(A)^2 \left( \frac{2(N-2)}{(N-1)^2} - \frac{4}{N-1}\right) \Bigg]  \\
    &  = \frac{C_N^2}4 \Bigg[ \Var \left( \int_A H_2(f(x)) dx \right)  \frac{N^2}{(N-1)^2}  - \vol{}(A)^2 \left(  \frac{2N}{(N-1)^2}\right) \Bigg].
\end{align}
If $A=M$, the last expression equals $0$ since the Orthogonality Assumption \cref{eq:Yorthogonality} implies \begin{equation}
    \Var \left( \int_M H_2(f(x)) dx \right) = \Var \left( \|f\|^2_{L^2(M)}\right) = \vol{}(M)^2 \frac2{N}.
\end{equation}
\end{proof}
\begin{lemma}\label{lem:gaussianCOV} Let us consider a standard Gaussian vector $\eta\sim\mathcal N(0,1_{N-1})$, and denote by $\eta_1$ its first component, with density $\varphi_1$. Let $H_m$ be the $m$-th Hermite polynomial. Then, for any $f\in\m C^1(\R)$ and $g\in\m C^0(\R^{N-1})$, the following hold: \begin{align}
    & \E [f(\eta_1)H_m(\eta_1)] = \E[f'(\eta_1)H_{m-1}(\eta_1)], \\
    & \E [\varphi_1(\alpha(\eta)) g(\eta)] = \frac{\sigma_N^{N-1}}{\sqrt{2\pi}} \E [ g(\sigma_N \eta)],
\end{align}
where $\alpha(\eta) = \|\eta\|\frac{u}{\sqrt{\frac{N}{{\vol{}(M)}}-u^2}}$, and $\sigma_N$ is as above.
\end{lemma}
\begin{proof}
    Both equations follow by standard change of variables arguments.
\end{proof}

\subsection{Fourth chaos: proof of \cref{p:lowerbound}}
Once the cancellation of the second chaos has been assessed, the next step is to proceed with the variance computation. This section is devoted to establishing fourth chaotic component variance asymptotics, implying in turn a lower bound for the variance of the whole functional. Recall first that
\begin{equation}
    \m E_u(\lebf|_A)[4]  = \int_{A}\I{\kop \lebf(x) \ge
 u \pok}[4] \, dx.
\end{equation}

We define $u_1$ and $u_2$ as the non-zero roots of the degree-three polynomial \begin{equation}\label{eq:degree3poly}
    H_3\left(u\sqrt{\vol{}(M)}\right)=u^3 \vol{}(M)^{3/2} - 3u \sqrt{\vol{}(M)}.
\end{equation}

We are now in a position to prove \cref{p:lowerbound}. The proof has the same structure as the proof of \cref{prop:secondchaos}, with some additional technicalities. We recommend reading first the proof of \cref{prop:secondchaos} in order to easily understand the next one. First, we provide a chaos decomposition of $\m E_u(\lebf|_A)[4]$, and then an explicit expression of its variance. For the second step, we defer some of the computations to the Appendix, see \cref{thm:cov4th}.

\begin{proof}[Proof of \cref{p:lowerbound}]
Thanks to \cref{lem:multimagictrick}, we can write down an explicit expression of $\I{\kop \lebf(x) \ge u \pok}[4]$, and of $\m E_u(\lebf|_A)[4]$, namely:
\begin{align} 
    \m E_u(\lebf|_A)[4] 
 & = C_N(4,4,u) \int_A \frac{H_{4}\tyu f(x)\uyt}{4!}\, dx \\
 & \qquad + C_N(4,2,u) \int_A  \fint_{S(Y(x)^\perp)} \frac{H_{2}(\gamma_x^Tv)}{2}
  \  \frac{H_{2}\tyu f(x)\uyt}{2} \ dv\, dx \\
  & \qquad + C_N(4,0,u)\int_A  \fint_{S(Y(x)^\perp)} \frac{H_{4}(\gamma_x^Tv)}{4!}
  \ dv \, dx,
\end{align}
where the coefficients \begin{equation}
    C_N(q,i,u) : = \E\bigg[ \I{\kop \lebf(x) \ge
 u \pok} H_{i}(f(x)) H_{q-i}(\gamma_x^{(1)})\bigg], \qquad (q,i) = (4,4),(4,2), (4,0),
\end{equation}
can be explicitly computed as in the proof of \cref{prop:secondchaos}: see \cref{eq:uno,eq:due}, applying \cref{lem:gaussianCOV}. Denoting $v=u\sqrt{\vol{}(M)}$, we find 
\begin{align}\label{eq:coeff4}
    & C_N(4,4,u) = - \frac{\sigma_N^{N-1}}{\sqrt{2\pi}} H_3(v) \xrightarrow{N\to\infty} J_4(v), \\
    & C_N(4,2,u) = - \frac{v \ \sigma_N^{N}}{\sqrt{2\pi}} \sqrt{\frac{N}{N-v^2}} H_2(v) \xrightarrow{N\to\infty} H_2(v) J_2(v) , \\
    & C_N(4,0,u) = C_N(u) (H_2(v)+2) \xrightarrow{N\to\infty}  J_2 (v)(H_2(v)+2) ,
\end{align}
where $C_N(u)$ is defined in \eqref{eq:coeff2}.

    From the above decomposition, it is clear that the variance of $\m E_u(\lebf|_A)[4]$ involves a linear combination of the variances and covariances investigated in \cref{thm:cov4th}, with coefficients \begin{equation}
        C_N(q,i,u)C_N(q',i',u), \qquad (q,i),(q',i') = (4,4),(4,2), (4,0).
    \end{equation}
    After some straightforward algebraic simplification, we obtain the following: 
    \begin{multline}
        \Var(\m E_u(\lebf|_A)[4])  = D_N(4,u) \int_A\int_Ak(x,y)^4 \, dx \, dy \\ + D_N(2,u) \int_A\int_Ak(x,y)^2 \, dx \, dy + D_N(0,u) \vol{}(A)^2,
    \end{multline} 
where \begin{align}
    D_N(4,u) & = \frac{C_N(4,4,u)^2}{4!}\left(1+o\left(\frac1N\right)\right) = \frac{J_4\left(u\sqrt{\vol{}(M)}\right)^2}{4!} \left(1+o\left(\frac1N\right)\right) ,\\
    D_N(2,u) & = O\left(\frac1N \right), \qquad D_N(0,u) = O\left(\frac1{N^2}\right).
\end{align}
We remark that $D_N(4,u)$ vanishes at levels $u=0,u_1,u_2$. Let us now suppose that $A=M=S^2$ and that $\lebf$ is a monochromatic uniform random wave of eigenvalue $\ell$, i.e., a random spherical harmonic. In this case, the asymptotics of \begin{equation}
    \int_{S^2} \int_{S^2} k(x,y)^q \, dx \, dy, \qquad N\to\infty, \quad q\ge 2
\end{equation}
are known, see \cite{MaWi2011defect,MaWi2014}. Then, as $N\to\infty$, we have 
    \begin{align} \label{eq:leading4}
        & D_N(4,u) \int_A\int_Ak(x,y)^4 \, dx \, dy 
        \sim c_4 \ \frac{\log N}{N^2},\\
        & D_N(2,u) \int_A\int_Ak(x,y)^2 \, dx \, dy \sim c_2 \frac1{N^2},
    \end{align}
    for some positive constant $c_2, c_4$, where $c_4$ is non-zero precisely if $u\notin \{0,u_1,u_2\}$. 
    This is enough to conclude the case $q=4$.

    A closer inspection of the previous arguments shows that for every $q\ge 3$, $q\neq 4$, and every level $u\in\R$, we have  \begin{equation}
        \Var(\m E_u(\lebf)[q]) \le 
        \frac{\cjboh_N(q,0,u)^2}{q!}
        \, \Var \left( \int_M H_q(f(x)) dx \right), \qquad N\to\infty;
    \end{equation}
    where $\cjboh_N(q,0,u)$ are as in \cref{t:wienerchaos}, hence square summable, hence bounded.
    This implies that, for a fixed chaos, the variance of the polyspectra of uniform random waves are comparable to the Gaussian ones.
\end{proof}

\section{Appendix: \cref{thm:cov2nd} and \cref{thm:cov4th}}

This Appendix is devoted to the calculations of the covariances between $H_2(f(x))$ and $\fint_{S^{N-2}} H_2(\gamma_x^\perp v)$, and $H_4(f(x))$, $\fint_{S^{N-2}} H_4(\gamma_x^\perp v)dv$, and $H_2(f(x))\fint_{S^{N-2}} H_2(\gamma_x^\perp v)dv$. We denoted the covariance function of $f$ by
\be 
k(x,z):=\E f(x)f(z), \quad \forall x,z\in M.
\ee
Under our working assumptions, we have $k(x,x)=1$. 
\begin{theorem}\label{thm:cov2nd}
\bega
        & \E \kop H_2(f(x))H_2(f(z))\pok=2k(x,z)^2,\\
        & \E \kop \fint_{S(Y(x)^\perp)} H_2(\gamma_x^\perp v) dv \cdot \fint_{S(Y(z)^\perp)} H_2(\gamma_z^\perp v) dv \pok 
        =\frac{2k(x,z)^2+2(N-2)}{(N-1)^2}, \\
        & \E \kop H_2(f(x)) \cdot  \fint_{S(Y(z)^\perp)} H_2(\gamma_z^\perp v) dv \pok  = \frac{2  -  2k(x,z)^2}{N-1}.
\eega
\end{theorem}

\begin{theorem}\label{thm:cov4th}
{\footnotesize\begin{align}
& \E \kop \textcolor{black} {H_4(f(x))H_4(f(z))}\pok  =  4!  k(x,z)^4 \nonumber \\
& \E \bigg\{ \textcolor{black}{H_2(f(x)) \fint_{S(Y(x)^\perp)} H_2(\gamma_x^\perp v) dv}  \ \cdot \ \textcolor{black}{H_2(f(z))   \fint_{S(Y(z)^\perp)} H_2(\gamma_z^T v) dv}\bigg\} \nonumber \\
& \hspace{3.75cm}  =  k(x,z)^4 \frac8{(N-1)^2} + k(x,z)^2 \frac{4(N-4)}{(N-1)^2} + \frac4{(N-1)^2}  \nonumber  \\       
& \E \kop \textcolor{black}{\fint_{S^{N-2}} H_4(\gamma_x^T v) dv} \cdot \textcolor{black}{\fint_{S^{N-2}} H_4(\gamma_z^T v) dv} \pok = \label{eq:covFourth2} k(x,z)^4 \cdot \frac{9\cdot 4!}{(N^2-1)^2} +  k(x,z)^2 \frac{6 \cdot 4! (N-2)}{(N^2-1)^2} + \frac{3\cdot 4! N(N-2)}{(N^2-1)^2}  \\
& \E \kop \textcolor{black}{H_4(f(x))} \cdot \textcolor{black}{H_2(f(z)) \fint_{S(Y(x)^\perp)} H_2(\gamma_z^T v) dv}     \pok = \frac{4!}{N-1} \left( k(x,z)^2 - k(x,z)^4\right)  \nonumber  \\
& \E \kop \textcolor{black}{H_4(f(x))} \cdot  \textcolor{black}{\fint_{S(Y(z)^\perp)} H_4(\gamma_z^T v) dv} \pok  = \frac{3\cdot 4!}{N^2-1} \left( 1 - 2k(x,z)^2 + k(x,z)^4\right)  \nonumber  \\
& \E \bigg\{ \textcolor{black}{H_2(f(x)) \fint_{S(Y(x)^\perp)} H_2(\gamma_x^\perp v) dv} \ \cdot \  \textcolor{black}{\fint_{S(Y(z)^\perp)} H_4(\gamma_z^T v) dv} \bigg\} \nonumber  \\
& \hspace{3.75cm} = \frac{4!(N-2)}{(N^2-1)(N-1))} - k(x,z)^2  \frac{4!(N-5)}{(N^2-1)(N-1))}  - \frac{4!3}{(N^2-1)(N-1))} k(x,z)^4 \nonumber 
    \end{align}}
\end{theorem}

\begin{proof}[Proof of \cref{thm:cov2nd}]
    The first equation follows by a standard property of Hermite polynomials, see e.g. \cite[Prop. 2.2.1]{nourdinpeccatibook}. Let us prove the second one. Given $Y(x)$, $Y(z)\in\R^{N}$ as in \eqref{eq:deff}, we define $\vartheta$ as the angle such that  \begin{align}\label{eq:costheta}
    \cos\vartheta :  = \left\langle \frac{Y(x)}{\|Y(x)\|},\frac{Y(z)}{\|Y(z)\|}\right\rangle.
\end{align}
As a consequence, $\cos\vartheta=k(x,z)$\footnote{Recall that if $M=S^2$, and if $Y=(Y_{\ell,1},\ldots, Y_{\ell,2\ell+1})$ has orthonormal components, then $\left\langle Y(x),Y(z) \right\rangle=\sum_{m} \overline{Y_{\ell m} (x)} Y_{\ell m}(z) = \frac{2\ell +1}{4\pi} P_\ell(\langle x,z\rangle) $, where $P_\ell$ is the $\ell$-th Legendre polynomial. Indeed, in this case $k(x,z)=P_\ell(\langle x,z\rangle)$.}.
Observe that the following double sphere integral only depends on $\vartheta$: 
\begin{equation}
     \fint_{S(Y(x)^\perp)} \fint_{S(Y(z)^\perp)}  |v^T w|^2 dv dw =: I(\vartheta).
\end{equation}
Observe that if $\gamma_x\randin Y(x)^\perp$ is a standard Gaussian random vector (in particular, it has rank $N-1$ and is supported on $Y(x)^\perp\cong \mathbb R^{N-1}$), and $v\randin S(Y(x)^\perp)$ is uniformly distributed on the sphere, then $\gamma_x$ and $ \|\gamma_x\| v$ share the same law. The same holds for $w$ and $\gamma_z$. Moreover, we have mutual independence among $v$, $w$, $ \|\gamma_x\|$,  and $ \|\gamma_z\|$. So,  
\begin{equation}
    I(\vartheta) = \frac{\E[|\langle\gamma_x,\gamma_z\rangle|^2]}{E[\|\gamma_x\|^2]E[\|\gamma_z\|^2]}.
\end{equation} 
Since $E[\|\gamma_x\|^2] = N-1$, we only need to compute $\E[|\langle\gamma_x,\gamma_z\rangle|^2]$. We can assume that \begin{align}
    \gamma_x = \begin{pmatrix}
        0,\\ \gamma_1\\ \gamma_2 \\ \vdots \\ \gamma_{N-1}
    \end{pmatrix} \text{ and } \gamma_z = \begin{pmatrix}
        -\sin \vartheta \xi_1\\ \cos\vartheta \xi_1\\\xi_2  \\ \vdots \\ \xi_{N-1}
    \end{pmatrix},
\end{align}
where $\vartheta$ is as in \eqref{eq:costheta}, and $\gamma_i$, $\xi_j$ are iid $\mathcal N(0,1)$. Therefore, \begin{align}
    \E[ |\langle\gamma_x , \gamma_z \rangle|^2] & = \E[(\cos\vartheta \gamma_1\xi_1+\gamma_2\xi_2 + \ldots + \gamma_{N-1}\xi_{N-1})^2] \\
    & = \cos^2\vartheta + N-2.
\end{align}
Putting all together, we get that \begin{equation}
    \frac12 \ \E \kop \fint_{S(Y(x)^\perp)} H_2(\gamma_x^\perp v) dv \cdot \fint_{S(Y(z)^\perp)} H_2(\gamma_z^\perp v) dv \pok  = I(\vartheta) = \frac{\cos^2\vartheta + N-2}{(N-1)^2},
\end{equation}
which proves the second equality. The third equation can be proved applying a similar strategy. Let $\vartheta$ be defined as before, and recall that we have defined the unit vector $\y(x)=Y(x) \|Y(x)\|^{-1}\in S^{N-1}$. Without loss of generality, we suppose that $\y(x)=e_1$. Then, \begin{equation}
        v = \begin{pmatrix}
        -\sin\vartheta \ v_1\\ \cos\vartheta\ v_1\\v_2  \\ \vdots \\ v_{N-1},
    \end{pmatrix}
    \end{equation} 
    where $(0,v_1,\ldots,v_{N-1})\in S(Y(x)^\perp)$. Since $v\in Y(z)^\perp$, we have that \begin{equation}
    \E \qwe f(x)    \gamma_z^T v   \ewq = \E \qwe \gamma^T \y(x)    \gamma^T v   \ewq = \E \qwe \gamma^T \y(x)    \gamma^T v   \ewq = \y(x)^Tv = - v_1 \sin\vartheta.
    \end{equation}
    This implies that 
    \begin{equation}
         \E \left( H_2(f(x)) dx \cdot  \fint_{S(Y(z)^T)} H_2(\gamma_z^T v) dv dz \right)  = \frac{2}{N-1}\left(1  -  k(x,y)^2\right).
    \end{equation}
\end{proof}


\begin{proof}[Proof of \cref{thm:cov4th}]
    We only prove \cref{eq:covFourth2}, regarding the covariance of the field $\fint_{S^{N-2}} H_4(\gamma_x^\perp v)dv$. The remaining equations can be proved in complete analogy with the argument of \cref{thm:cov2nd}. Since $v\in Y(x)^\perp$ and $w\in Y(z)^\perp$, we may write $\gamma_x^T v = \gamma^T v$ and $\gamma_z^T v = \gamma^T v$. Then, \begin{align}
        \E\kop H_4(\gamma_x^T v) H_4(\gamma_z^T v) \pok & = 4! \fint \fint (v^T w)^4 dv dw \\
        & = 4! \frac{\E \kop(\cos\theta \gamma_1\xi_1 + \gamma_2\xi_2 + \ldots + \gamma_{N-1}\xi_{N-1})^4\pok}{\E \kop \|\gamma_x\|^4 \pok \E \kop \|\gamma_z\|^4 \pok}
    \end{align}
    where $\gamma_i,\xi_i$ are iid $\mathcal N(0,1)$. Observe that if $\gamma\sim \mathcal N(0,1_{N-1})$ we have $\E \|\gamma\|^4 = N^2 - 1$. Moreover, for two independent Gaussian vectors $\hat \gamma,\hat \xi\sim \mathcal N(0,1_{N-2})$, applying \cref{lem:4thScalarProduct} in the last step, we have \begin{align}
        \E & \kop(\cos\theta \gamma_1\xi_1 + \gamma_2\xi_2 + \ldots + \gamma_{N-1}\xi_{N-1})^4\pok  = \E \kop(\cos\theta \gamma_1\xi_1 + \langle \hat\gamma,\hat\xi \rangle)^4\pok  \\
        & =  9 \cos^4\theta + 6 \cos^2\vartheta  \E \kop\langle \hat\gamma,\hat\xi \rangle^2\pok + \E \kop\langle \hat\gamma,\hat\xi \rangle^4\pok \\
        & = 9 \cos^4\theta + 6 \cos^2\vartheta  (N-2) + 3N(N-2).
    \end{align}
    Recalling that $\cos\theta = k(x,y)$, we conclude.
\end{proof}
\begin{lemma}\label{lem:4thScalarProduct}
    Let us consider two independent Gaussian vectors $\xi$, $\eta\sim\mathcal N(0,\mathbbm{1}_N)$, $N\ge 1$. Then \begin{equation}
        \E[(\langle \xi, \eta \rangle)^4] = 3N(N+2).
    \end{equation}
\end{lemma}
\begin{proof} Enough to observe that $a_N:=\E[(\langle \xi, \eta \rangle)^4]$, for $N\ge1$, obeys the recursive relation \begin{equation} \begin{cases}
    a_N &= 9 + 6(N-1) + a_{N-1} = 6 N + 3 + a_{N-1}, \qquad N\ge 2,\\
    a_1 &= 9.
\end{cases}
    \end{equation}
\end{proof}

{\teal 

}

{ 

\section{Appendix: Proof of Theorem \ref{t:nonmale} and Ancillary Statements}\label{A:SHproofs}

\subsection{On Wick Products}\label{ss:AAWick}

Fix an integer $N \ge 1$. For every $q \ge 1$ and every vector 
$(i_1,\dots,i_q) \in [N]^q$, we define the \emph{multiplicity vector}
\begin{equation}\label{e:multiplicity}
  \alpha(i_1,\dots,i_q) = (\alpha_1,\dots,\alpha_N) \in \mathbb{N}^N
\end{equation}
by letting $\alpha_\ell$ denote, for each $\ell = 1,\dots,N$, the number of
indices $r \in \{1,\dots,q\}$ such that $i_r = \ell$; in particular,
$\sum_{\ell=1}^N \alpha_\ell = q$.\footnote{For instance, if $N=4$, $\alpha(1,2) = (1,1,0,0)$, $\alpha(2,2,4) = (0,2,0,1)$, $\alpha(4,2,4,1,3) = (1,1,1,2)$, and so on and so forth.  } We also set $\alpha(i_1,...,i_q)! := \prod_{\ell=1}^N \alpha_\ell!$, and note that, if $\alpha(i_1,...,i_q) = \alpha(j_1,...,j_p)$, then necessarily $p=q$.

\smallskip

\begin{remark}\label{r:combi}
Given $(i_1,...,i_q)\in [N]^q$, there exist exactly
    $$
    \binom{q}{\alpha(i_1,...,i_q)}:= \frac{q!}{\alpha(i_1,...,i_q)!}
    $$
    vectors $(j_1,...,j_q)$ such that $\alpha(i_1,...,i_q) = \alpha(j_1,...,j_q)$.
\end{remark}

\smallskip

Let $\gamma = (\gamma^1,...,\gamma^N)\sim \mathcal{N}(0, \mathbb{I}_N)$. According e.g. to \cite[Definition 3.1]{JansonBook}, given $(i_1,...,i_q)\in [N]^q$, the {\it Wick product} of the random variables $\gamma^{i_1},...,\gamma^{i_q}$ is given by
\begin{equation}\label{e:wickdef}
    :\gamma^{i_1}\cdots \gamma^{i_q}:\quad := (\gamma^{i_1}\cdots \gamma^{i_q})[q]= \prod_{\ell=1}^N H_{\alpha_\ell}(\gamma^\ell),
\end{equation}
where $(\alpha_1,...,\alpha_N) = \alpha(i_1,...,i_q)$ is the multiplicity vector associated with $(i_1,...,i_q)$, and $H_m$ is the $m$th Hermite polynomial---see Definition \ref{d:wiener}. The following relation is easily verified: for every $(i_1,...,i_q)\in [N]^q$ and $(j_1,...,j_p)\in [N]^p$,
\begin{equation}\label{e:cowick}
\mathbb{E}[:\gamma^{i_1}\cdots \gamma^{i_q}:\,\times\,:\gamma^{j_1}\cdots \gamma^{j_p}:] = \begin{cases}
\alpha(i_1,...,i_q)!, & \mbox{if}\,\,\, \alpha(i_1,...,i_q) =\alpha(j_1,...,i_p)  ,\\
0,  & \mbox{otherwise}.
\end{cases}
\end{equation}
In particular, combining \eqref{e:cowick} with Remark \ref{r:combi} yields the isometric formula \eqref{e:innertenso}.

\medskip

\subsection{Two Lemmas}\label{ss:2L}
For $N \ge 2$, let $\gamma = (\gamma^1,\dots,\gamma^N)\sim \mathcal{N}(0,\mathbb{I}_N)$, 
let $E(\gamma)$ be square-integrable, and keep the notation \eqref{e:wienerexp} and \eqref{e:tensorq}---\eqref{e:hompqS} (from Section \ref{ss:furthnot}). The following lemma plays an important role throughout the paper.

\begin{lemma}[\bf Averaging trick]\label{l:magictrick} 
Let above notation prevail and let $E : \R^N\to \R$ be a 0-homogeneous mapping (in the sense of Definition \ref{d:0hom}) such that $E(\gamma)$ is square-integrable. Then, 
\bega \label{e:winwin}
\fint_{S^{N-1}}P_q(v) dv = 0, \quad \mbox{for all}\,\,q\geq 1.
\eega
In particular, ${\rm Tr}\, K_2 = 0$, and therefore $K_2\in \mathbb{T}_0(2,N)$.
\end{lemma}
\begin{proof} One has that 
\bega\notag
\fint_{S^{N-1}}P_q(v) dv&=\fint_{S^{N-1}}\E\kop E(\gamma) H_q\tyu \langle \gamma ,v\rangle \uyt \pok dv
\\
&=\fint_{S^{N-1}}\fint_{S^{N-1}}\E\kop E\tyu w \|\gamma\|\uyt H_q\tyu \langle w\|\gamma\| ,v\rangle \uyt \pok dvdw
\\
&=\fint_{S^{N-1}}E\tyu w \uyt\E\kop  H_q\tyu \langle w,\gamma\rangle \uyt \pok dw=0,
\eega
where we have used the 0-homogeneity of $E(\cdot)$. 
To deal with the second part of the statement, we first observe that, by definition,
\begin{equation}\label{e:rrep2}
P_2(v) = 2\, v^T K_2v, \quad v\in S^{N-1},
\end{equation}
so that the identity \eqref{e:winwin} (in the case $q=2$) yields that 
$$ 0 = \fint_{S^{N-1}}P_2(v) dv = 2\fint_{S^{N-1}}v^T K_2 v\,  dv  =\frac{2}{N}{\rm Tr}\, K_2 ,$$ 
and therefore the desired conclusion.
\end{proof}

\medskip 

The following statement is the combinatorial key to prove Theorem \ref{t:nonmale}.
\begin{lemma}\label{l:wicknowick}
For every $q\geq 1$ and for every traceless tensor $K\in \mathbb{T}_0(q,N)$, one has that 
\begin{equation}\label{e:zz}
\sum_{(i_1,...,i_q)\in [N]^q} K(i_1,...,i_q) \, :\gamma_{i_1}\cdots\gamma_{i_q}:\,\,\quad  =\!\!\! \sum_{(i_1,...,i_q)\in [N]^q} K(i_1,...,i_q) \, \gamma_{i_1}\cdots\gamma_{i_q} \,, \end{equation}
that is, the homogeneous polynomial associated with $K$ coincides with the polynomial obtained by replacing products with Wick products.
\end{lemma}
\begin{proof} The statement is trivial for $q=1,2$, so that one can focus on the case $q\geq 3$. We need some further notation. We denote by $\mathcal{P}_{\leq 2}[q]$ the collection of all partitions of $[q]$ having blocks at most of size $2$. Given $\pi\in \mathcal{P}_{\leq 2}[q]$, we write $S_\pi$ for the collection of all singletons of $\pi$, and observe that, writing $|A|$ for the number of blocks in $A\subseteq\pi$, one has that
$$
|\pi| = |S_\pi| +|S_\pi^c|, \quad \mbox{and}\quad q = |S_\pi| +2|S_\pi^c|. 
$$
We also write $\hat{0}$ for the element of $\mathcal{P}_{\leq 2}[q]$ whose elements are all singletons (so that $\hat{0} =S_{\hat{0}}$). The crucial element in our proof is the following consequence of \cite[Corollary 3.17]{JansonBook} (see also \cite[Chapter 6]{PecTaqqBook}): for every $(i_1,...,i_q)\in [N]^q$,
$$
\gamma_{i_1}\cdots \gamma_{i_q} = \sum_{\pi\in \mathcal{P}_{\leq 2}[q]} \prod_{\{a,b\}\in \pi} \delta_{i_a i_b} : \prod_{\{u\}\in S_\pi} \gamma_{i_u} : ,
$$
where $\delta_{ij}$ is Kronecker's symbol. 
Exploiting this identity, one sees that the difference between the right-hand side and the left-hand side of \eqref{e:zz} is given by 
$$
\sum_{(i_1,..,i_q)\in [N]^q} K(i_1,...,i_q) \sum_{\substack{\pi\in\mathcal{P}_{\le 2}[q] \\ \pi\neq \hat{0}}} \prod_{\{a,b\}\in \pi} \delta_{i_a i_b} : \prod_{\{u\}\in S_\pi} \gamma_{i_u} :,
$$
and some elementary algebraic steps yield that this quantity equals in turn
$$
\sum_{\substack{\pi\in\mathcal{P}_{\le 2}[q] \\ \pi\neq \hat{0}}} \,\,\sum_{(\ell_1,...,\ell_{|S_\pi|})\in [N]^{|S_\pi|} }\,\, : \prod_{t=1}^{|S_\pi|} \gamma_{\ell_t}:  \sum_{(j_1,...,j_{|S^c_\pi|})\in [N]^{|S^c_\pi|}} K(\ell_1,...,\ell_{|S_\pi|}, j_1,j_1,...,j_{|S_\pi^c|},j_{|S_\pi^c| }) =0,
$$
since $K\in \mathbb{T}_0(q,N)$ by assumption. This proves the desired conclusion.
\end{proof}

\medskip 

\subsection{End of the proof of Theorem \ref{t:nonmale}}


Formula \eqref{e:notnaiveatall} is a consequence of Lemma \ref{l:wicknowick}, yielding that, for every traceless $A\in \mathbb{T}_0(q,N)$, the expectation
$$
Z :=\mathbb{E}\left[ \sum_{(i_1,...,i_q)\in [N]^q} K^{\rm (TL)}_q(i_1,...,i_q) \, U^{i_1}\cdots U^{i_q} \times \sum_{(i_1,...,i_q)\in [N]^q} A_q(i_1,...,i_q) \, U^{i_1}\cdots U^{i_q}\right]
$$
equals exactly
\begin{eqnarray*}
&& \frac{1}{\mathbb{E}[\|\gamma\|^{2q}]} \mathbb{E}\left[ \sum_{(i_1,...,i_q)\in [N]^q} K^{\rm (TL)}_q(i_1,...,i_q) \, \gamma^{i_1}\cdots \gamma^{i_q} \times \sum_{(i_1,...,i_q)\in [N]^q} A(i_1,...,i_q) \, \gamma^{i_1}\cdots \gamma^{i_q}\right]\\
&& = \frac{1}{\mathbb{E}[\|\gamma\|^{2q}]} \mathbb{E}\left[ \sum_{(i_1,...,i_q)\in [N]^q} \!\!\!\!\!K^{\rm (TL)}_q(i_1,...,i_q) \, :\gamma^{i_1}\cdots \gamma^{i_q}: \times\!\!\! \sum_{(i_1,...,i_q)\in [N]^q} \!\!\! \!\!A(i_1,...,i_q) \, :\gamma^{i_1}\cdots \gamma^{i_q}:\right]\\
&& = \frac{1}{\mathbb{E}[\|\gamma\|^{2q}]} \mathbb{E}\left[ \sum_{(i_1,...,i_q)\in [N]^q} K_q(i_1,...,i_q) \, :\gamma^{i_1}\cdots \gamma^{i_q}: \times\!\!\!\! \sum_{(i_1,...,i_q)\in [N]^q} A(i_1,...,i_q) \, :\gamma^{i_1}\cdots \gamma^{i_q}:\right]\\
&&= \frac{1}{\mathbb{E}[\|\gamma\|^{2q}]} \mathbb{E}\left[ E(\gamma)\times \sum_{(i_1,...,i_q)\in [N]^q} A(i_1,...,i_q) \, \gamma^{i_1}\cdots \gamma^{i_q}\right],
\end{eqnarray*}
where the second equality follows from the isometric relation \eqref{e:innertenso}, as well as from the orthogonal decomposition $K_q =K^{\rm (TL)}_q+ K^{\rm (Tr)}_q$, see \eqref{e:tracelessD}. Since $E(\gamma) = E(U)$ by $0$-homogeneity, the previous relation yields
$$
Z = c_{q,N}^{-1} \, \mathbb{E}\left[E(U)\times \sum_{(i_1,...,i_q)\in [N]^q} A_q(i_1,...,i_q) \, U_{i_1}\cdots U_{i_q}\right],
$$
which is the desired conclusion. The last part of the statement is a direct consequence of Lemma \ref{l:magictrick}.

\bigskip


}


\begin{remark}[\emph{On scalar products}] For the homogeneous polynomial $x^\a=x_1^{\a_1} \dots x_N^{\a_N}$, represented by a tensor $K_\a$, we have:
\begin{gather} 
K_\a(i_1,\dots,i_q)=\frac{\a_1!\dots \a_N!}{q!}\begin{cases}
    1, & \text{if $\#\tyu\kop \ell\in \{1,\dots,q\}|i_\ell=j\pok\uyt=\a_j\ \forall j\in \{1,\dots,N\}$}
    \\
    0, & \text{otherwise}
\end{cases}
\\
\|K_\a\|^2_{\R^{[N]^q}}=\frac{\a_1!\dots \a_N!}{q!}
=\|x^\a\|^2_{L^2(S(\C^N))}\binom{q+N-1}{q}, \\ \E\kop |:\gamma^\a:|^2\pok=
\prod_{j=1}^N\E\kop |H_{\a_j}(\gamma_j)|^2\pok=\a_1!\dots \a_N!
\end{gather}
In the last equation, we are using that, on homogeneous polynomials, the Wick product coincides with the chaotic projection, so that $:\gamma_1^{\a_1}\dots \gamma_N^{\a_N}:=H_{\a_1}(\gamma_1)\dots H_{\a_N}(\gamma_N)$. For all these scalar products, the monomials form an orthogonal basis, so that from the above relation we deduce that
\be 
\E\kop |X|^2\pok={q!}\|K\|^2_{\R^{[N]^q}},
\ee
whenever $X=\sum_{(i_1,...,i_q)\in [N]^q} K(i_1,...,i_q) \, :\gamma_{i_1}\cdots\gamma_{i_q}:$. This means that the mapping defined by $K\mapsto X$ is an isometry from the space of symmetric tensors $(\mathbb{T}(q,N),{q!}\langle\cdot,\cdot \rangle_{\R^{[N]^q}})$, endowed with the component-wise scalar product rescaled by $q!$, that is (since $K$ and $H$ are symmetric),
\be 
q!\langle K,H \rangle_{\R^{[N]^q}}=\sum_{ i_1,\dots,i_q\in \{1,\dots, N\}}\sum_{\tau\in \Sigma_q}K(i_{\tau(1)},\dots,i_{\tau(q)})H(i_{\tau(1)},\dots,i_{\tau(q)}),
\ee 
to the $q$-th chaos space $ (C_q, \langle\cdot,\cdot \rangle_{L^2(\P)})$. Hence, the Wick product $p(x)\mapsto :p(\gamma):$ identifies both spaces isometrically with the space $\R[x_1,\dots,x_N]_{(q)}$ of homogeneous polynomials, endowed with the Bombieri-Weyl (or Kostlan) norm, defined by

\[
\|x^\a\|^2_{Kostlan}=\|x^\a\|^2_{Bombieri-Weyl}
=\|x^\a\|^2_{L^2(S(\C^N))}\binom{q+N-1}{q}.
\]
\end{remark}


\medskip


\bibliographystyle{abbrv}
\bibliography{bs.bib}

\end{document}